\documentclass[reqno]{amsproc}
\usepackage{amssymb}
\usepackage{color}
\usepackage{euscript}
\usepackage{mathrsfs}

\DeclareMathOperator{\supp}{supp}

\newcommand*{\borel}[1]{\mathfrak B(#1)}

\newcommand*{\cbb}{\mathbb C}

\newcommand*{\D}{\mathrm d}

\newcommand*{\E}{\mathrm{e}}

\newcommand*{\Ge}{\geqslant}
\newcommand*{\hh}{\mathcal H}
\newcommand*{\is}[2]{\langle#1,#2\rangle}
\newcommand*{\kk}{\mathcal K}

\newcommand*{\jd}[1]{{\mathscr N}(#1)}
\newcommand*{\lambdab}{\boldsymbol\lambda}
\newcommand*{\Le}{\leqslant}
\newcommand*{\nbb}{\mathbb N}
\newcommand*{\ob}[1]{{\mathscr R}(#1)}

\newcommand*{\ogr}[1]{\boldsymbol B(#1)}

\newcommand*{\qbb}{\mathbb Q}
\newcommand*{\rbb}{{\mathbb R}}

\newcommand*{\zbb}{\mathbb Z}

\begin{document}
   \title[The hyperbolic cosine
transform and its applications] {The hyperbolic cosine
transform and its applications to composition
operators}
   \author[J.\ Stochel]{Jan Stochel}
   \address{Instytut Matematyki, Uniwersytet Jagiello\'nski,
ul.\ {\L}ojasiewicza 6, 30-348 Kra\-k\'ow, Poland}
   \email{Jan.Stochel@im.uj.edu.pl}
   \author[J.\ B.\ Stochel]{Jerzy Bart{\l}omiej Stochel}
\address{Faculty of Applied Mathematics,
AGH University of Krakow, Al. Mickiewicza 30, 30-059
Krak\'ow, Poland}
   \email{stochel@agh.edu.pl}
  \thanks{The research of the first author was
supported by the National Science Center (NCN) Grant
OPUS No. DEC-2021/43/B/ST1/01651.}
   \subjclass[2020]{Primary 43A35, 47B33, 47B20;
Secondary 47B90, 44A60}
   \keywords{The hyperbolic cosine transform,
composition operators on $L^2$-spaces, affine symbols,
cosubnormal operators}
   \begin{abstract}
In this paper we characterize hyperbolic cosine
transforms of (positive) Borel measures $\nu$ in terms
of exponential convexity (Bernstein's terminology).
The case of compactly supported measures $\nu$ is also
considered. All of this is then applied to (bounded)
composition operators $C_{T,\rho}\colon f \mapsto f
\circ T$ on $L^2(\rbb^\kappa,\mu_{\rho})$ with affine
symbols $T=A+a$, where $\D \mu_{\rho} (x) = \rho(x) \D
x$, $\rho(x)= \psi(\|x\|)^{-1}$, $\psi$ is a
continuous positive real valued function and
$\|\cdot\|$ is the Euclidean norm on $\rbb^{\kappa}$.
The main result states that the map $\rbb^{\kappa} \ni
a \mapsto C_{I+a,\rho}$ is continuous in the strong
operator topology and has cosubnormal values if and
only if $\psi$ is the hyperbolic cosine transform of a
compactly supported Borel measure ($I$ is the identity
transformation). The case of affine symbols $T$ that
are not translations is also discussed.
   \end{abstract}
   \maketitle

\numberwithin{equation}{section}
\newtheorem*{thmI*}{Theorem}
   \section{Introduction} 
   Denote by $\rbb$ and $\rbb_+$ the sets of all real
and nonnegative real numbers, respectively. Given a
function $\psi\colon \rbb\to \rbb_+$ and a (positive)
Borel measure $\nu$ on $\rbb_+$ such~that
   \begin{align} \label{ch-tr-1}
\psi(x)=\int_{\rbb_+} \cosh(xu) \D \nu(u), \quad x\in
\rbb,
   \end{align}
we call $\psi$ the {\em hyperbolic cosine transform}
of the measure $\nu$ (clearly, $\nu(\rbb_+)<\infty$).
We show that a function $\psi\colon \rbb\to \rbb_+$ is
the hyperbolic cosine transform of a Borel measure on
$\rbb_+$ if and only if it is continuous, even and
exponentially convex (see Theorem~\ref{mmainth}). In
modern terminology, the notion of exponential
convexity attributed to Bernstein is understood as
positive definiteness with respect to the Hermitian
involution on the additive group $\rbb$ (see
Section~\ref{Sec.2.2}). The present paper is mainly
devoted to the study of the hyperbolic cosine
transforms of compactly supported Borel measures on
$\rbb_+$. In addition to the properties mentioned
above, such transforms are characterized by
exponential growth (see Theorem~\ref{maintheorem}).
Moreover, if $\psi$ is the hyperbolic cosine transform
of a nonzero Borel measure $\nu$ on $\rbb_+$, then
$\psi$ is analytic, the limit $\lim_{x\to\infty}
\displaystyle{\frac{\psi^{\prime}(x)}{\psi(x)}}$
exists in $[0,\infty]$ and $\lim_{x\to\infty}
\displaystyle{\frac{\psi^{\prime}(x)}{\psi(x)}} = \sup
\supp \nu$, where $\supp \nu$ stands for the closed
support of $\nu$ (see Theorem~\ref{limsupq-1}).

The motivation for studying the cosine transform comes
from the research on seminormal composition operators
on $L^2$ spaces with matrix symbols, which was
initiated in \cite{Ml78,Sto90} (see also \cite{jsjs11}
for the Laplace density case and \cite{BJJS15,BJJS18}
for the unbounded case). Subsequently, this topic was
continued in \cite{Da-St97} in the context of affine
symbols with a matrix linear part. The question of
cosubnormality\footnote{The subnormal operators are
just restrictions of the normal operators to their
invariant subspaces, and the cosubnormal operators are
the adjoints of the subnormal operators.} (or
subnormality, depending on whether the measure under
consideration is finite or not) is even more
challenging. In this paper, we deal with (bounded)
composition operators $C_{T,\rho}$ on
$L^2(\mu_{\rho})$ defined by
   \begin{align*}
C_{T,\rho}(f) = f \circ T, \quad f \in
L^2(\mu_{\rho}),
   \end{align*}
where
   \begin{enumerate}
   \item[$\bullet$] $T=A+a$, $A$ is an
invertible linear transformation of $\rbb^\kappa$ and
$a\in \rbb^\kappa$,
   \item[$\bullet$] $\D \mu_{\rho} = \rho \D V_\kappa$, where
$V_\kappa$ is the Lebesgue measure on $\rbb^\kappa$,
$\rho(\cdot)= \psi(\|\cdot\|)^{-1}$, $\psi\colon
\rbb\to (0,\infty)$ is a continuous even function and
$\|\cdot\|$ is the Euclidean norm on $\rbb^{\kappa}$.
   \end{enumerate}
According to \cite{Da-St97}, there is extremely little
room for cosubnormality of composition operators with
such symbols. Roughly speaking, the only potential
candidates are composition operators $C_{T,\rho}$ with
translation symbols, i.e., $T=I+a$, where $I$ is the
identity transformation on $\rbb^{\kappa}$ and $a\in
\rbb^{\kappa}$. The main result of this paper,
Theorem~\ref{wkw-uch}, shows that there is a close
relationship between the cosubnormality of such
operators and the hyperbolic cosine transform. The
following is an excerpt from this theorem.
   \begin{thmI*}
Under the above assumptions, the following are
equivalent{\em :}
   \begin{enumerate}
   \item[(i)] $C_{I+a,\rho}$ is
a bounded cosubnormal operator on $L^2(\mu_{\rho})$
for every $a\in \rbb^{\kappa}$, and the map
$\rbb^{\kappa} \ni a \mapsto C_{I+a,\rho}$ is
continuous in the strong operator topology,
   \item[(ii)] $\psi$ is the hyperbolic cosine transform of a
compactly supported Borel measure on $\rbb_+$.
   \end{enumerate}
   \end{thmI*}
It is worth noting here that the pioneers of this type
of research were Embry and Lambert, who discovered a
relationship between the subnormality of a weighted
translation semigroup and the Laplace-Stieltjes
transform (see \cite{Em-Lam77B}; see also
\cite{Em-Lam77A,Em-Lam91}). Research in this area,
taking into account both the single-parameter and
multi-parameter cases, is still being conducted (see,
e.g., \cite{Pha-Sho19,Pha-Sho20,Pha-Sho23}).

One of the main tools used in the proof of
Theorem~\ref{wkw-uch} is the theorem which states that
the function
   \begin{align*}
\rbb \ni x \longmapsto \cosh\big(\sqrt{\xi x^2 +
\eta}\/\big)
   \end{align*}
is exponentially convex for all $\xi, \eta \in \rbb_+$
(see \cite{Me-Kum76}; see also \cite{Kat16,Kat17}).
This is a special case of the BMV conjecture proved by
Stahl in \cite{Sta13}.

The paper is organized as follows. Sections~\ref{Sec.2.1} and
\ref{Sec.2.2} are introductory and are devoted to
highlighting the relationship between the hyperbolic cosine
transform, the Stieltjes moment problem (see
Lemma~\ref{wkw-supp-zw} and Theorem~\ref{cosh-Sti}) and the
positive definiteness in the $*$-semigroup sense. In
Section~\ref{Sec.3} we establish necessary and sufficient
conditions for a real-valued function on $\rbb$ to be the
hyperbolic cosine transform of a Borel measure on $\rbb_+$
(see Theorem~\ref{mmainth}). The case of hyperbolic cosine
transforms of compactly supported measures on $\rbb_+$ is
treated in Theorem~\ref{maintheorem}. In Section~\ref{Sec.4},
we give an explicit formula for the supremum of the closed
support of $\nu$ written in terms of the logarithmic
derivative of $\psi$ (see Theorem~\ref{limsupq-1}).
Section~\ref{Sec.3.1} provides basic facts about the
composition operators under consideration.
Section~\ref{Sec.3.2} contains the main result of this paper,
Theorem~\ref{wkw-uch}, which relates the hyperbolic cosine
transform to cosubnormality when $T=I+a$. There are also
several results leading from the cosubnormality of
composition operators with more general symbols $T=A+a$ to
the hyperbolic cosine transform (see
Propositions~\ref{impi-ch} and \ref{wkw-ukh}). In
Section~\ref{Sec.3.3}, we give an example of cohyponormal
composition operators $\{C_{I+a,\rho}\} _{a\in
\rbb^{\kappa}}$ for which there exists $\varepsilon \in
(0,\infty)$ such that $C_{I+a,\rho}$ is not cosubnormal for
any $a\in \rbb^{\kappa}$ with $0 < \|a\| < \varepsilon$ (see
Theorem~\ref{blis-zer} and Proposition~\ref{blis-zur}). In
Section~\ref{Sec.3.4}, we relate our main result to the
Embry-Lambert theorem (see \cite[Theorem~2.2]{Em-Lam77B}),
and show that considering unbounded cosubnormal composition
operators $C_{T,\rho}$ opens up the possibility of studying
hyperbolic cosine transforms of Borel measures with unbounded
supports, essentially enlarging the associated class of
density functions (see Example~\ref{invdid}). The paper is
accompanied by an appendix containing another proof of
Theorem~\ref{maintheorem}.

\numberwithin{equation}{subsection}
\newtheorem{thm}{Theorem}[subsection]
\newtheorem{cor}[thm]{Corollary}
\newtheorem{lem}[thm]{Lemma}
\newtheorem{pro}[thm]{Proposition}
\newtheorem{opq}[thm]{Problem}
\newcounter{step}
\newtheorem*{step}{Step \thestep}
\theoremstyle{remark}
\newtheorem{rem}[thm]{Remark}
\newtheorem{exa}[thm]{Example}
\newtheorem*{remm}{Remark}
\theoremstyle{definition}
\newtheorem{elu}[thm]{Example}
   \section{The hyperbolic cosine transform}
   \subsection{\label{Sec.2.1}Connections with the Stieltjes moment problem}
The fields of real and complex numbers are denoted by
$\rbb$ and $\cbb$, respectively. Set $\nbb = \{1,2,3,
\ldots\}$ and $\zbb_+=\{0,1,2,3, \ldots\}$. If no
ambiguity arises, the characteristic function of a
subset $Y$ of a set $X$ is denoted by $\chi_Y$. We
write $\borel{X}$ for the $\sigma$-algebra of all
Borel subsets of a topological Hausdorff space $X$. A
Borel measure on $X$ is said be {\em compactly
supported} if it vanishes off a compact subset of $X$.
By a {\em Radon} measure on $X$, we mean a Borel
measure on $X$ which is finite on compact sets and
inner regular. The closed support of a Borel measure
$\mu$ on $X$ is denoted by $\supp \mu$ (it always
exists for Radon measures). In this paper we consider
mostly finite Borel measures on locally compact
Hausdorff spaces whose all open sets are
$\sigma$-compact. As a consequence, such measures are
Radon measures, meaning in particular that they are
regular (see e.g., \cite[Theorem 2.18]{rud87}). Given
a real-valued function $\psi$ on $\rbb$ (or on a
subinterval of $\rbb$) and $k\in \zbb_+,$ we denote by
$\psi^{(k)}$ the $k$th derivative of $\psi$ provided
it exists. If $k=1,2$, we also use the notation
$\psi^{\prime}$ and $\psi^{\prime\prime}$.

A sequence $\{a_n\}_{n=0}^\infty \subseteq \rbb_+$ is said to
be a {\em Stieltjes moment sequence} if there exists a Borel
measure $\mu$ on $\rbb_+$ such that
   \begin{align} \label{rep1}
a_n = \int_{\rbb_+} x^n \D \mu(x), \quad n \in \zbb_+;
   \end{align}
such $\mu$ is called a {\em representing measure} of
$\{a_n\}_{n=0}^\infty$. If $\mu$ in \eqref{rep1} is
unique, the Stieltjes moment sequence
$\{a_n\}_{n=0}^\infty$ is called {\em determinate}.
Otherwise, we call it {\em indeterminate}. It is worth
recalling that each Stieltjes moment sequence which
has a compactly supported representing measure is
determinate (cf.\ \cite{fug}). We refer the reader to
\cite{sh-tam,Akh65,fug,b-ch-r,sim98} for the
foundations of theory of moment problems.

The following lemma shows that the hyperbolic cosine
transform is analytic and that its Taylor's
coefficients at zero are related to the Stieltjes
moment problem.
   \begin{lem} \label{wkw-supp-zw}
Let $\psi\colon \rbb \to \rbb_+$ be the hyperbolic cosine
transform of a Borel measure $\nu$ on $\rbb_+$. Then
   \begin{enumerate}
   \item[(i)] $\psi$ has the power series expansion
   \begin{align} \label{cosh-b}
\psi(x) = \sum_{n=0}^{\infty} \frac{\gamma_n}{(2n)!}
\, x^{2n}, \quad x\in \rbb,
   \end{align}
where $\gamma_n=\int_{\rbb_+} u^{2n} \D \nu(u) <
\infty$ for all $n \in \zbb_+$,
   \item[(ii)] the derivative $\psi^{\prime}(x)$ of $\psi$
at $x$ is given by
   \begin{align} \label{cosh-a}
\psi^{\prime}(x) := \int_{\rbb_+} u \sinh(x u) \D
\nu(u), \quad x\in \rbb.
   \end{align}
   \end{enumerate}
Moreover, the following hold{\em :}
   \begin{enumerate}
   \item[(a)] there is only one
Borel measure $\nu$ on $\rbb_+$ satisfying
\eqref{ch-tr-1},
   \item[(b)] the sequence $\{\gamma_n\}_{n=0}^{\infty}$ in
\eqref{cosh-b} is unique and is a determinate
Stieltjes moment sequence with a representing measure
$\nu \circ \omega$, where $\omega\colon \rbb_+ \to
\rbb_+$ is given by $\omega(u)=\sqrt{u}$ for $u\in
\rbb_+$ and $(\nu \circ
\omega)(\varDelta)=\nu(\omega(\varDelta))$ for
$\varDelta \in \borel{\rbb_+}$.
   \end{enumerate}
   \end{lem}
   \begin{proof}
(i) \& (ii) Applying \cite[Theorem~1.27]{rud87}, we
obtain
   \allowdisplaybreaks
   \begin{align} \notag
\int_{\rbb_+} \cosh(xu) \D \nu(u) & = \int_{\rbb_+}
\sum_{n=0}^{\infty} \frac{x^{2n}}{(2n)!} u^{2n} \D
\nu(u)
   \\ \label{sosh-tr-1}
&= \sum_{n=0}^{\infty} \frac{x^{2n}}{(2n)!}
\int_{\rbb_+} u^{2n} \D \nu(u), \quad x\in \rbb.
   \end{align}
This implies \eqref{cosh-b}. Using similar reasoning
and the oddness of $\sinh,$ we get
   \begin{align} \label{gru-gru}
\int_{\rbb_+} u \sinh(xu) \D \nu(u) =
\sum_{n=0}^{\infty} \frac{\gamma_{n+1}}{(2n+1)!} \,
x^{2n+1}, \quad x\in \rbb.
   \end{align}
Term-by-term differentiation of the power series in
\eqref{cosh-b} yields
   \begin{align*}
\psi^{\prime}(x) = \sum_{n=0}^{\infty}
\frac{\gamma_{n+1}}{(2n+1)!} x^{2n+1}
\overset{\eqref{gru-gru}}= \int_{\rbb_+} u \sinh(xu) \D
\nu(u), \quad x\in \rbb.
   \end{align*}

We now prove the ``moreover'' part. Clearly, by the
uniqueness of Taylor's coefficients for power series,
the sequence $\{\gamma_n\}_{n=0}^{\infty}$ in
\eqref{cosh-b} is unique. Applying the measure
transport theorem \cite[Theorem~1.6.12]{Ash} yields
   \begin{align*}
\gamma_n=\int_{\rbb_+} u^{2n} \D \nu(u) =
\int_{\rbb_+} t^n \D \nu\circ \omega(t), \quad n\in
\zbb_+.
   \end{align*}
Therefore, $\{\gamma_n\}_{n=0}^{\infty}$ is a
Stieltjes moment sequence with a representing measure
$\nu\circ \omega$. In view of \eqref{cosh-b}, the
radius of convergence of the power series
$\sum_{n=0}^{\infty} \frac{\gamma_n}{(2n)!} \, x^{n}$
is equal to infinity, so by the Cauchy radius formula
(see \cite[Theorem~III.1.3]{Con78})
   \begin{align*}
\lim_{n\to \infty} \sqrt[n]{\frac{\gamma_n}{(2n)!}}
=0.
   \end{align*}
This implies that there exists $M\in (0,\infty)$ such
that (see Footnote~\ref{fut-1})
   \begin{align*}
\gamma_n \Le M (2n)!, \quad n\in \zbb_+.
   \end{align*}
Hence, by \cite[Proposition~1.5]{sim98} (or by
Carleman's criterion \cite[Corollary~4.5]{sim98} and
Stirling's approximation theorem
\cite[Theorem~9.7.1]{sim15}), the Stieltjes moment
sequence $\{\gamma_n\}_{n=0}^{\infty}$ is determinate
(it is also determinate as a Hamburger moment
sequence). This proves (b). It remains to justify (a).
Suppose that $\nu_1$ and $\nu_2$ are Borel measures on
$\rbb_+$ satisfying \eqref{ch-tr-1}. Then, in view of
the above discussion, the measures $\nu_1\circ \omega$
and $\nu_2\circ \omega$ are representing measures of
the determinate Stieltjes moment sequence
$\{\gamma_n\}_{n=0}^{\infty}$. Thus $\nu_1\circ \omega
= \nu_2\circ \omega$, which implies that
$\nu_1=\nu_2$. This completes the proof.
   \end{proof}
Below, we show that there is a one-to-one
correspondence between hyperbolic cosine transforms
and Stieltjes moment sequences of factorial growth at
infinity.
   \begin{thm}\label{cosh-Sti}
The identity \eqref{cosh-b} establishes a one-one
correspondence between hyperbolic cosine transforms
$\psi$ of Borel measures on $\rbb_+$ and Stieltjes
moment sequences $\{\gamma_n\}_{n=0}^{\infty}$ such
that\/\footnote{\;\label{fut-1}Note that $\lim_{n\to
\infty} \sqrt[n]{\frac{\gamma_n}{(2n)!}} =0$ if and
only if for every $\varepsilon \in (0,\infty)$ there
exists $M\in (0,\infty)$ such that $\gamma_n \Le M
\varepsilon^n (2n)!$ for all $n\in \zbb_+$.}
$\lim_{n\to \infty} \sqrt[n]{\frac{\gamma_n}{(2n)!}}
=0$. Moreover, if a Borel measure $\nu$ on $\rbb_+$
satisfies \eqref{ch-tr-1}, then $\nu \circ \omega$ is
a representing measures of
$\{\gamma_n\}_{n=0}^{\infty}$ and, vice versa, if
$\mu$ is a representing measures of
$\{\gamma_n\}_{n=0}^{\infty}$, then the measure $\nu$
on $\rbb_+$ defined by
$\nu(\varDelta)=\mu(\omega^{-1}(\varDelta))$ for
$\varDelta \in \borel{\rbb_+}$ satisfies
\eqref{ch-tr-1}, where $\omega$ and $\nu \circ \omega$
are as in Lemma~{\em \ref{wkw-supp-zw}}.
   \end{thm}
   \begin{proof}
Apply Lemma~\ref{wkw-supp-zw} and its proof together
with the identity \eqref{sosh-tr-1} which holds
regardless of wheth\-er the integrals $\int_{\rbb_+}
\cosh(xu)\D \nu(u)$, $x\in \rbb$, are finite or not.
   \end{proof}
   \begin{rem} \label{Taylor-gam}
It is also worth noting that, in view of
Lemma~\ref{wkw-supp-zw} and Theorem~\ref{cosh-Sti}, the
condition (i) of Lemma~\ref{wkw-supp-zw} is equivalent to
stating that the function $\psi$ is analytic,
$\psi^{(2n+1)}(0)=0$ for all $n\in \zbb_+$,
$\{\psi^{(2n)}(0)\}_{n=0}^{\infty}$ is a Stieltjes moment
sequence and $\lim_{n\to \infty}
\sqrt[n]{\frac{\psi^{(2n)}(0)}{(2n)!}} =0$. This is because
$\psi^{(2n)}(0)=\gamma_n$ for $n\in \zbb_+.$
   \hfill{$\diamondsuit$}
   \end{rem}
We now give an example of the hyperbolic cosine
transform of a Borel measure on $\rbb_+$ which is not
compactly supported.
   \begin{exa}
Let $\{\gamma_n\}_{n=0}^{\infty}$ be the sequence
defined by $\gamma_n = n!$ for $n\in \zbb_+$. Since
$\frac{n!}{(2n)!} \Le \frac{1}{n^n}$ for all $n\in
\nbb$, we see that $\lim_{n\to \infty}
\sqrt[n]{\frac{\gamma_n}{(2n)!}} =0$. It is well
known~that
   \begin{align*}
\gamma_n = \int_{\rbb_+} x^n \E^{-x}\D x,\quad n\in
\zbb_+.
   \end{align*}
This means that $\{\gamma_n\}_{n=0}^{\infty}$ is a
Stieltjes moment sequence with the representing
measure $\mu$ given by $\mu(\varDelta) =
\int_{\varDelta} \E^{-x}\D x$ for $\varDelta \in
\borel{\rbb_+}$. Then the Borel measure $\nu$ on
$\rbb_+$ defined as in Theorem~\ref{cosh-Sti} takes
the form
   \begin{align*}
\nu(\varDelta) = \mu(\omega^{-1}(\varDelta)) =
\int_{\rbb_+} \chi_{\varDelta}(\sqrt{x}\,) \E^{-x} \D
x = 2 \int_{\varDelta} u \E^{-u^2} \D u, \quad
\varDelta \in \borel{\rbb_+}.
   \end{align*}
A routine computation shows that the hyperbolic cosine
transform $\psi$ of the measure $\nu$ is of the form
(see \cite[integral 4.\ in 3.562 and integral 1.\ in
8.250]{Gr-Ryz07})
   \begin{align*}
\psi(x)=\int_{\rbb_+} \cosh(xu) \D \nu(u) & = 2
\int_{\rbb_+} u \cosh(xu)\E^{-u^2} \D u
   \\
& = 1 + \frac{\sqrt{\pi}}{2} x \,
\varPhi\Big(\frac{x}{2}\Big) \E^{\frac{x^2}{4}}, \quad
x \in \rbb,
   \end{align*}
where $\varPhi\colon \rbb \to \rbb$ is the so-called
error function (which is odd) defined by
   \begin{align*} \tag*{$\diamondsuit$}
\varPhi(x) = \frac{2}{\sqrt{\pi}} \int_{0}^{x}
\E^{-t^2} \D t, \quad x\in \rbb.
   \end{align*}
   \end{exa}
   \subsection{\label{Sec.2.2}Positive definite functions on
$*$-semigroups} Let $(S,+,*)$ be an abelian
$*$-semigroup with a commutative addition $+$,
involution $*$ and the zero element denoted by $0$. We
call a function $\psi\colon S \to \cbb$ {\em positive
definite} if
   \begin{align} \label{eq1}
\sum_{j,k=1}^n \psi(s_j^* + s_k) \lambda_j
\bar\lambda_k \Ge 0, \quad n \in \nbb, \,
\{s_j\}_{j=1}^n \subseteq S, \,
\{\lambda_j\}_{j=1}^n\subseteq \cbb.
   \end{align}
If $\psi(s^*)= \overline{\psi(s)}$ for all $s\in S$
and the inequality in \eqref{eq1} holds for all $n \in
\nbb$, $\{s_j\}_{j=1}^n \subseteq S$ and
$\{\lambda_j\}_{j=1}^n\subseteq \cbb$ such that
$\sum_{j=1}^n \lambda_j=0$, then $\psi$ is called {\em
conditionally positive definite}.

In Bernstein's terminology (see
\cite[Sect.~15]{Ber28}), positive definite functions
on the abelian $*$-semigroup $(\rbb,+,x^*=x)$ are
called {\em exponentially convex}. Note that if
$\psi\colon \rbb \to \cbb$ is exponentially convex,
then either $\psi =0$ or $\psi(\rbb) \subseteq
(0,\infty)$. Indeed, substituting $n=1$, $s_1=\frac
x2$ and $\lambda_1=1$ into \eqref{eq1} gives
$\psi(x)\Ge 0$ for all $x\in \rbb$, so $\psi(\rbb)
\subseteq \rbb_+$. If $\psi(x_0)=0$ for some $x_0\in
\rbb$, then \cite[p.\ 69, {\bf 1.8}]{b-ch-r} yields
   \begin{align*}
0 \Le \psi(x)=\psi\Big(\Big(\frac {x_0}2\Big)^* +
\Big(x-\frac {x_0}2 \Big)\Big) \Le \sqrt{\psi(x_0)}
\sqrt{\psi(2x-x_0)} = 0, \quad x \in \rbb,
   \end{align*}
which shows that $\psi=0$.

We need the following characterizations of
exponentially convex functions.
   \begin{lem} \label{kiedy-cosub-cosh}
Let $\psi\colon \rbb \to \rbb_+$ be a function and
$\{\xi_m\}_{m=1}^{\infty} \subseteq \rbb$ be a
sequence with $\lim_{m\to \infty} \xi_m=-\infty$. Then
the following conditions are equivalent{\em :}
   \begin{enumerate}
   \item[(i)] $\psi$ is exponentially convex,
   \item[(ii)] the function
$\rbb_+ \ni x \mapsto \psi(\xi_m + x)\in \rbb_+$ is
positive definite on the $*$-semigroup
$(\rbb_+,+,x^*=x)$ for every $m\in \nbb$.
   \end{enumerate}
Moreover, if $\psi$ is continuous and
$\{s_k\}_{k=1}^{\infty} \subseteq \nbb$, then {\em
(i)} is equivalent to each of the following
conditions{\em :}
   \begin{enumerate}
   \item[(iii)] the function $\zbb_+ \ni n
\mapsto \psi\big(\xi + nt\big) \in \rbb_+$ is positive
definite on the $*$-semi\-group $(\zbb_+,+,n^*=n)$ for
all $\xi,t \in \rbb$,
   \item[(iv)] the function $\zbb_+ \ni n
\mapsto \psi\big(\xi_m + \frac{n}{k\cdot s_k}\big) \in
\rbb_+$ is positive definite on the $*$-semi\-group
$(\zbb_+,+,n^*=n)$ for all $k, m \in \nbb$.
   \end{enumerate}
   \end{lem}
   \begin{proof}
That the conditions (i) and (ii) are equivalent can
easily be verified. Since the implications
(i)$\Rightarrow$(iii) and (iii)$\Rightarrow$(iv) are
obvious, it suffices to prove that (iv) implies (ii)
whenever $\psi$ is continuous. Denote by $\qbb_+$ the
additive semigroup of all nonnegative rational
numbers. Fix $m\in \nbb$. Let
$\{q_j\}_{j=1}^l\subseteq \qbb_+$ be a finite
sequence. Then there exists $k\in \nbb$ and a sequence
$\{p_j\}_{j=1}^l \subseteq \zbb_+$ such that
$q_j=\frac {p_j}{k\cdot s_k}$ for $j=1, \ldots, l$.
This implies that
   \begin{align*}
\sum_{i,j=1}^{l} \psi(\xi_{m} + (q_i^* + q_j))
\lambda_i \bar \lambda_j = \sum_{i,j=1}^{l}
\psi\Big(\xi_{m} + \frac{p_i^* + p_j}{k\cdot
s_k} \Big) \lambda_i \bar \lambda_j \Ge 0, \quad
\{\lambda_j\}_{j=1}^l\subseteq \cbb.
   \end{align*}
Hence, the function $\qbb_+ \ni x \mapsto \psi(\xi_{m}
+ x)\in \rbb_+$ is positive definite on the
$*$-semigroup $(\qbb_+,+,x^*=x)$ for every $m\in
\nbb$. By the continuity of $\psi$, we deduce that the
function $\rbb_+ \ni x \mapsto \psi(\xi_{m} + x)\in
\rbb_+$ is positive definite on the $*$-semigroup
$(\rbb_+,+,x^*=x)$ for every $m\in \nbb$. This
completes the proof.
   \end{proof}
Continuous exponentially convex functions can be
described as follows (see \cite[Theorem~5.5.4]{Akh65}
and its proof; see also \cite[Theorem~21]{Widd46}).
   \begin{thm}[The representation theorem]\label{Bernst}
A function $\psi\colon \rbb \to \rbb_+$ is of the form
   \begin{align} \label{juden}
\psi(x) = \int_{\rbb} \E^{xu} \D\sigma(u), \quad x \in
\rbb,
   \end{align}
where $\sigma$ is a Borel measure on $\rbb$, if and
only if $\psi$ is continuous and exponentially convex.
Moreover, the Borel measure $\sigma$ satisfying {\em
\eqref{juden}} is unique.
   \end{thm}
The following theorem, which is essentially due to
Mehta and Kumar (see \cite[Sect.~2]{Me-Kum76}), is a
special case of the BMV conjecture. Their proof relies
heavily on the idea of using a perturbation series.
The BMV conjecture itself was proven by Stahl in
\cite{Sta13}. Stahl's proof is based on brilliant
considerations related to Riemann surfaces of
algebraic functions. Recently, Katsnelson provided a
purely ``matrix'' proof of the BMV conjecture for the
case of $2 \times 2$ matrices (see \cite{Kat17}).
   \begin{thm}[{\cite[Theorem~1]{Kat16}}] \label{Katz}
Let $\xi \in (0,\infty)$ and $\eta\in \rbb_+$. Then
the function $\varphi_{\xi\eta}\colon \rbb \to \rbb_+$
defined by
   \begin{align*}
\varphi_{\xi\eta}(x) = \cosh\big(\sqrt{\xi x^2 +
\eta}\/\big), \quad x\in \rbb,
   \end{align*}
is exponentially convex.
   \end{thm}
Note that if $\xi > 0$ and $\eta\Ge 0$, then the function
$\varphi_{\xi\eta}$ is of exponential growth (see
Theorem~\ref{maintheorem}) and the function $x \mapsto
\sqrt{\xi x^2 + \eta}$ is not exponentially convex.
   \subsection{\label{Sec.3}Characterizations of the hyperbolic cosine transform} In
this section we give necessary and sufficient
conditions for a positive-valued function on the real
line $\rbb$ to be the hyperbolic cosine transform of a
Borel measure on $\rbb_+$.
   \begin{thm} \label{mmainth}
For $\psi\colon \rbb \to \rbb_+$, the following
conditions are equivalent{\em :}
   \begin{enumerate}
   \item[(i)] $\psi$ is the hyperbolic
cosine transform of a Borel measure $\nu$ on $\rbb_+$,
   \item[(ii)] $\psi$ has the following properties{\em :}
   \begin{enumerate}
   \item[(ii-a)] $\psi$ is exponentially convex,
   \item[(ii-b)] $\psi$ is continuous,
   \item[(ii-c)] $\psi(x)=\psi(-x)$ for every $x\in
\rbb$.
   \end{enumerate}
   \end{enumerate}
Moreover, if {\em (ii)} holds, then there exists only
one Borel measure $\nu$ on $\rbb_+$ satisfying
\eqref{ch-tr-1}, the sequence
$\{\gamma_n\}_{n=0}^{\infty}$ in \eqref{cosh-b} is
unique and is a determinate Stieltjes moment sequence
with a representing measure $\nu \circ \omega$, where
$\omega$ and $\nu \circ \omega$ are as in Lemma~ {\em
\ref{wkw-supp-zw}}, and $($with the convention that
$\sup\supp \nu=0$ if $\supp{\nu}=\varnothing)$
   \begin{gather} \label{coshtr3new}
\sup\supp \nu = \lim_{n\to\infty} \gamma_{n}^{1/2n}.
   \end{gather}
   \end{thm}
   \begin{proof}
(i)$\Rightarrow$(ii) Notice that
   \allowdisplaybreaks
   \begin{align*}
\sum_{j,k=1}^n \cosh((x_j^*+x_k)u) \lambda_j
\bar\lambda_k &= \sum_{j,k=1}^n \cosh(x_ju)
\cosh(x_ku) \lambda_j \bar\lambda_k
    \\
& \hspace{7.6ex} + \sum_{j,k=1}^n \sinh(x_ju)
\sinh(x_ku) \lambda_j \bar\lambda_k
   \\
& = \bigg|\sum_{j=1}^n \cosh(x_ju) \lambda_j\bigg|^2 +
\bigg|\sum_{j=1}^n \sinh(x_ju) \lambda_j\bigg|^2 \Ge 0
   \end{align*}
for all $n \in \nbb$, $\{x_j\}_{j=1}^n \subseteq
\rbb$, $\{\lambda_j\}_{j=1}^n\subseteq \cbb$ and $u\in
\rbb_+$. This and \eqref{ch-tr-1} imply \mbox{(ii-a)}.
The condition \mbox{(ii-b)} is a direct consequence of
Lemma~\ref{wkw-supp-zw}(i). Finally, the condition
\mbox{(ii-c)} follows from the fact that $\cosh$ is an
even function. Therefore (ii) is valid.

(ii)$\Rightarrow$(i) Since \mbox{(ii-a)} and
\mbox{(ii-b)} hold, Theorem~\ref{Bernst} implies that
there exists a Borel measure $\sigma$ on $\rbb$
satisfying \eqref{juden}. However, by \mbox{(ii-c)},
$\psi$ is even, so we have
   \begin{align*}
\psi(x) & = \frac{1}{2}(\psi(x) + \psi(-x))
\overset{\eqref{juden}}= \int_{\rbb} \cosh(xu) \D
\sigma(u)
   \\
& = \int_{(-\infty,0)} \cosh(xu) \D \sigma(u) +
\sigma(\{0\}) + \int_{(0,\infty)} \cosh(xu) \D
\sigma(u)
   \\
& = \int_{\rbb_+} \cosh(xu) \D \nu(u), \quad x \in
\rbb,
   \end{align*}
where $\nu$ is the Borel measure on $\rbb_+$ defined
by
   \begin{align*}
\nu(\varDelta) = \sigma(-(\varDelta \cap (0,\infty)))
+ \sigma(\{0\})\cdot \delta_0(\varDelta) +
\sigma(\varDelta \cap (0,\infty)), \quad \varDelta \in
\borel{\rbb_+}.
   \end{align*}
(As usual, $\delta_0$ denotes the Borel probability
measure on $\rbb_+$ concentrated on $\{0\}$.) This
means that $\psi$ is of the form \eqref{ch-tr-1}.

The ``moreover'' part except \eqref{coshtr3new} is a
direct consequence of Lemma~\ref{wkw-supp-zw}. Using
\cite[Exercise 4(e), p.\ 71]{rud87}, we obtain
   \begin{align*}
\lim_{n\to\infty} \gamma_{n}^{1/2n} =
\|f\|_{L^{\infty}(\nu)} = \sup\supp \nu,
   \end{align*}
where $f \colon \rbb_+ \to \rbb_+$ is the identity
function. This completes the proof.
   \end{proof}
   \begin{rem}
It is worth mentioning that Theorem~\ref{mmainth}
(more precisely, the implication (ii)$\Rightarrow$(i))
was derived from from Theorem~\ref{Bernst}. A similar
reasoning can be used to deduce Theorem~\ref{Bernst}
from Theorem~\ref{mmainth}, so the two are logically
equivalent. In other words, if $\psi\colon \rbb \to
\rbb_+$ is the hyperbolic cosine transform of a Borel
measure $\nu$ on $\rbb_+$, then there exists a Borel
measure $\sigma$ on $\rbb$ (expressed in terms of
$\nu$) such that \eqref{juden} holds, and {\em vice
versa}.
   \hfill{$\diamondsuit$}
   \end{rem}
In this paper, we mainly deal with hyperbolic cosine
transforms of measures with compact supports. This is
because the composition operators studied in the
second part of the article, induced by hyperbolic
cosine transforms of {\em a priori} arbitrary
measures, are bounded if and only if these measures
have compact supports. The next theorem answers the
question when a given positive-valued function is the
hyperbolic cosine transform of a measure with compact
support. We give a proof relying on the representation
theorem (see Theorem~\ref{Bernst}). Another proof
based on the characterization of the Laplace transform
is provided in Appendix~A. Yet another proof can be
obtained by using the Berg-Maserick-Szafraniec theorem
(see \cite{Sza,B-M} or \cite[Theorem~ 4.2.5]{b-ch-r}).
   \begin{thm} \label{maintheorem}
For $\psi\colon \rbb \to \rbb_+$, the following
conditions are equivalent{\em :}
   \begin{enumerate}
   \item[(i)] $\psi$ is the hyperbolic
cosine transform of a compactly supported Borel
measure $\nu$ on $\rbb_+$,
   \item[(ii)] $\psi$ has the following properties{\em :}
   \begin{enumerate}
   \item[(ii-a)] $\psi$ is exponentially convex,
   \item[(ii-b)] $\psi$ is continuous,
   \item[(ii-c)] $\psi(x)=\psi(-x)$ for every $x\in
\rbb$,
   \item[(ii-d)] there exist  $a,b\in \rbb_+$ such
that
   \begin{align} \label{coshtr2}
\psi(x) \Le a \E^{b|x|}, \quad x \in \rbb,
   \end{align}
   \end{enumerate}
   \item[(iii)] there exists a Stieltjes moment sequence
$\{\gamma_n\}_{n=0}^{\infty}$ such that \eqref{cosh-b}
holds and
   \begin{align} \label{lim-sup-2}
\lim_{n\to\infty} \gamma_{n}^{1/n} <\infty.
   \end{align}
   \end{enumerate}
Moreover, if any of the conditions {\em (i)}-{\em
(iii)} holds, then
   \begin{gather} \label{coshtr3}
\min\big\{b\in \rbb_+ \colon \text{\eqref{coshtr2}
holds for some $a\in \rbb_+$}\big\} = \max\supp \nu =
\lim_{n\to\infty} \gamma_{n}^{1/2n}.
   \end{gather}
   \end{thm}
   \begin{proof}
(i)$\Rightarrow$(ii) According to
Theorem~\ref{mmainth}, the conditions
\mbox{(ii-a)}-\mbox{(ii-c)} are satisfied. By (i),
there exists $c\in \rbb_+$ such that
$\nu((c,\infty))=0$. Noting that
   \begin{align} \label{coshle}
\cosh(xu) = \cosh(|x|u) \Le \cosh(c|x|) \Le \E^{c|x|},
\quad x\in \rbb, \, u \in [0,c],
   \end{align}
we deduce from \eqref{ch-tr-1} that \eqref{coshtr2}
holds with $a=\nu([0,c])$ and $b=c$, so \mbox{(ii-d)}
is satisfied.

(ii)$\Rightarrow$(i) In view of Theorem~\ref{mmainth},
$\psi$ is the hyperbolic cosine transform of a Borel
measure $\nu$ on $\rbb_+$. Set $c_0:= \sup\supp \nu$
and
   \begin{align} \label{defb0}
\text{$b_0:=\inf\big\{b\in \rbb_+ \colon
\text{\eqref{coshtr2} holds for some $a\in
\rbb_+$}\big\}$.}
   \end{align}
By \mbox{(ii-d)}, $b_0\in \rbb_+$. Our goal is to
prove that $c_0<\infty$. Suppose, to the contrary,
that $c_0=\infty$. It follows from the definition of
$b_0$ that there exist $b\in [b_0,c_0)$ and $a\in
\rbb_+$ such that \eqref{coshtr2} holds. In turn, by
the definition of $c_0$, there exist $c, \varepsilon
> 0$ such that
   \begin{align} \label{war1}
b < c - \varepsilon \text{ and } \nu(U_{\varepsilon})
> 0,
   \end{align}
where $U_{\varepsilon}:=(c-\varepsilon,
c+\varepsilon)$. Since $\cosh$ is a strictly
increasing function on $\rbb_+$, we get
   \begin{align*}
\psi(x) \overset{\eqref{ch-tr-1}} \Ge
\int_{U_{\varepsilon}} \cosh(xu) \D \nu(u) \Ge
\nu(U_{\varepsilon}) \cosh((c-\varepsilon)x), \quad
x\in \rbb_+.
   \end{align*}
Therefore, we have
   \begin{align*}
\frac 12 \E^{(c-\varepsilon)x} \Le
\cosh((c-\varepsilon)x) \Le
\frac{\psi(x)}{\nu(U_{\varepsilon})}
\overset{\eqref{coshtr2}}\Le
\frac{a}{\nu(U_{\varepsilon})} \E^{bx}, \quad x \in
\rbb_+.
   \end{align*}
Hence, by \eqref{war1}, we arrive at a contradiction.
This shows that $\nu$ is compactly supported, which
means that (i) is satisfied

(i)$\Rightarrow$(iii) This implication and the second
equality in \eqref{coshtr3} follow directly from
Lemma~\ref{wkw-supp-zw} and the equality
\eqref{coshtr3new} in Theorem~\ref{mmainth}.

(iii)$\Rightarrow$(i) Use Theorem~\ref{cosh-Sti} and
the equality \eqref{coshtr3new} in
Theorem~\ref{mmainth}.

It remains to prove the first equality in
\eqref{coshtr3}. Assume that condition (ii) holds.
Using the same reasoning as in the proof of
implication (ii)$\Rightarrow$(i) with $c_0$ in place
of $c$, it can be shown that $b_0 \Ge c_0$. Applying
\eqref{coshle} to $c=c_0$, we see that \eqref{coshtr2}
holds with $a=\nu([0,c_0])$ and $b=c_0$, so $b_0=c_0$
and \eqref{defb0} holds with ``min'' in place of
``inf''. This completes the proof.
   \end{proof}
   \begin{rem}
It follows from the moreover part of Theorem~
\ref{maintheorem} that if $\psi$ is a continuous even
and exponentially convex function, then $\psi$ is
bounded if and only if $\nu((0,\infty))=0$, or
equivalently if and only if $\psi$ is constant. By
Lemma~\ref{kiedy-cosub-cosh}, a function $\psi\colon
\rbb \to \rbb_+$ is exponentially convex if and only
if for every $\xi\in \rbb$, the function $x \to
\psi(\xi + x)$ is positive definite on the
$*$-semigroup $(\rbb_+,+,x^*=x)$. In turn, a
well-known characterization of the Laplace transform
states that a function $\psi\colon \rbb_+ \to \rbb$ is
the Laplace transform of a Borel measure on $\rbb_+$
if and only if $\psi$ is bounded continuous and
positive definite on the $*$-semigroup
$(\rbb_+,+,x^*=x)$ (see \cite[Corollary~
4.4.5]{b-ch-r}). This means that the only function
$\psi\colon \rbb_+ \to \rbb$ which is simultaneously
the Laplace transform of a Borel measure on $\rbb_+$
and the hyperbolic cosine transform of a compactly
supported Borel measure on $\rbb_+$ is a constant
function.
   \hfill{$\diamondsuit$}
   \end{rem}
The following result is a version of
Theorem~\ref{maintheorem} needed in this paper (in
fact, Theorems~\ref{maintheorem} and
\ref{maintheorem-sq} are easily seen to be
equivalent).
   \begin{thm} \label{maintheorem-sq}
For $\varphi\colon \rbb_+ \to \rbb_+$, the following
conditions are equivalent{\em :}
   \begin{enumerate}
   \item[(i)] there exists a compactly supported
Borel measure $\nu$ on $\rbb_+$ such that
   \begin{align} \label{coshtr-sq-2}
\varphi(x)= \int_{\rbb_+} \cosh(\sqrt{x}\, u) \D
\nu(u), \quad x\in \rbb_+,
   \end{align}
   \item[(ii)] $\varphi$ has the following properties{\em :}
   \begin{enumerate}
   \item[(ii-a)] the function $x \longmapsto \varphi(x^2)$
is exponentially convex,
   \item[(ii-b)] $\varphi$ is continuous,
   \item[(ii-c)] there exist  $a,b\in \rbb_+$ such
that
   \begin{align} \label{coshtr2-sq}
\varphi(x) \Le a \E^{b\sqrt{x}}, \quad x \in \rbb_+,
   \end{align}
   \end{enumerate}
   \item[(iii)] there exists a Stieltjes moment sequence
$\{\gamma_n\}_{n=0}^{\infty}$ such that
\eqref{lim-sup-2} holds and
   \begin{align} \label{psisum-sq}
\varphi(x) = \sum_{n=0}^{\infty}
\frac{\gamma_{n}}{(2n)!} x^{n}, \quad x\in \rbb_+.
   \end{align}
   \end{enumerate}
Moreover, if {\em (ii)} holds, then there is only one
Borel measure $\nu$ on $\rbb_+$ satisfying
\eqref{coshtr-sq-2}, the sequence
$\{\gamma_n\}_{n=0}^{\infty}$ in \eqref{psisum-sq} is
unique and is a determinate Stieltjes~ moment sequence
with a representing measure $\nu \circ \omega$, and
   \begin{gather*}
\min\Big\{b\in \rbb_+ \colon \text{\eqref{coshtr2-sq}
holds for some $a\in \rbb_+$}\Big\} = \max\supp \nu =
\lim_{n\to\infty} \gamma_{n}^{1/2n}.
   \end{gather*}
   \end{thm}
   \begin{proof}
One can apply Theorems~ \ref{mmainth} and
\ref{maintheorem} to the function $\psi\colon \rbb \to
\rbb_+$ defined by $\psi(x)= \varphi(x^2)$ for $x \in
\rbb$.
   \end{proof}
   \begin{cor}
If $\varphi\colon \rbb_+ \to \rbb_+$ satisfies the
condition {\em \mbox{(ii)}} of Theorem {\em
\ref{maintheorem-sq}}, then $\varphi$ extends uniquely
to the entire function $\widehat\varphi$ defined by
   \begin{gather*}
\widehat\varphi(z) = \sum_{n=0}^{\infty}
\frac{\gamma_{n}}{(2n)!} z^n, \quad z\in \cbb,
   \end{gather*}
where $\{\gamma_n\}_{n=0}^{\infty}$ is as in the
statement {\em (iii)} of Theorem {\em
\ref{maintheorem-sq}}.
   \end{cor}
Now we give an example of a function which is the hyperbolic
cosine transform of a compactly supported Borel measure on
$\rbb_+$.
   \begin{exa}
Define the function $\psi\colon \rbb\to \rbb$ by
   \begin{align*}
\psi(x) =
   \begin{cases}
\displaystyle\frac{\sinh x}{x} & \text{if } x\neq 0,
   \\[1ex]
1 & \text{if } x=0,
   \end{cases}
   \qquad x\in \rbb.
   \end{align*}
It is easily seen that
   \begin{align*}
\psi(x) = \int_{[0,1]} \cosh(x\, u) \D u, \quad x\in
\rbb.
   \end{align*}
The function $\psi$ has the following power series
expansion:
   \begin{align*}
\psi(x) = \sum_{n=0}^{\infty} \frac{x^{2n}}{(2n+1)!},
\quad x\in \rbb.
   \end{align*}
In particular, we have
   \begin{align*}
\frac{\sinh \sqrt{x}}{\sqrt{x}} = \int_{[0,1]}
\cosh(\sqrt{x}\, u) \D u, \quad x\in (0,\infty),
   \end{align*}
and
   \begin{align*} \tag*{$\diamondsuit$}
\frac{\sinh \sqrt{x}}{\sqrt{x}} = \sum_{n=0}^{\infty}
\frac{x^{n}}{(2n+1)!}, \quad x\in (0,\infty).
   \end{align*}
   \end{exa}
   \subsection{\label{Sec.4}The logarithmic derivative}
   It follows from Theorem~\ref{maintheorem} that the
hyperbolic cosine transform $\psi$ of a Borel measure
$\nu$ has at most exponential growth if and only if
the closed support of the measure $\nu$ is compact. In
this section we will provide yet another criterion for
the compactness of the closed support of the measure
$\nu$ formulated in terms of the logarithmic
derivative of the transform $\psi$ (see
Theorem~\ref{limsupq-1} below).

Before stating the aforementioned result, we
define the class $\mathscr{H}$. For $k\in \nbb$,
we denote by $\mathscr{H}_k$ the set of all
entire functions $\varPhi$ on $\cbb$ of the form
   \begin{align}  \label{dub-sta}
\varPhi(z) = \alpha_0 + \sum_{n=k}^\infty \alpha_n
z^n, \quad z \in \cbb,
   \end{align}
where $\alpha_0\in \rbb_+$, $\alpha_k > 0$ and
$\alpha_n\in \rbb_+$ for all integers $n \Ge
k+1$. Set
   \allowdisplaybreaks
   \begin{align*}
\mathscr{H} & = \bigcup_{k=1}^{\infty} \mathscr{H}_k,
   \\
\mathscr{H}_0 & = \{\varPhi \in \mathscr{H} \colon
\varPhi(0) > 0\}
   \\
\mathscr{H}_{\bullet} & = \bigcap_{k=0}^{\infty}
\Big\{\psi \in \mathscr{H}\colon \psi^{(k)}(0) > 0
\Big\},
   \\
\mathscr{H}_{2\bullet}& = \bigcap_{k=0}^{\infty}
\Big\{\psi \in \mathscr{H}\colon \psi^{(2k)}(0)
> 0 \text{ and } \psi^{(2k+1)}(0) = 0\Big\}.
   \end{align*}
Note that $\varPhi(\rbb_+) \subseteq \rbb_+$ and
$\lim_{\rbb_+ \ni t\to \infty} \varPhi(t)=\infty$ for
every $\varPhi \in \mathscr{H}$. Since $\varPhi\in
\mathscr{H}$ is uniquely determined by its values on
$\rbb_+$ (resp., $\rbb$), we will also write $\psi \in
\mathscr{H}$ when $\psi$ is a function on $\rbb_+$
(resp., $\rbb$) that has an extension to the entire
function $\widehat \psi \in \mathscr{H}$. Clearly
$\exp \in \mathscr{H}_{\bullet}$.
   \begin{thm} \label{limsupq-1}
Suppose that $\psi\colon \rbb\to\rbb_+$ is the
hyperbolic cosine transform of a nonzero Borel measure
$\nu$ on $\rbb_+$. Then
   \begin{enumerate}
   \item[(i)] if $\supp \nu \neq \{0\}$, then $\psi
\in \mathscr{H}_{2\bullet}$,
   \item[(ii)] the limit $\lim_{x\to\infty}
\displaystyle{\frac{\psi^{\prime}(x)}{\psi(x)}}$ exists in
$[0,\infty]$ and
   \begin{align} \label{heliou4-1}
\lim_{x\to\infty}
\displaystyle{\frac{\psi^{\prime}(x)}{\psi(x)}} = \sup_{x
\in (0,\infty)}
\displaystyle{\frac{\psi^{\prime}(x)}{\psi(x)}} = \sup
\supp \nu,
   \end{align}
   \item[(iii)] the following conditions are equivalent{\em :}
   \begin{enumerate}
   \item[(a)] $\nu$ is compactly supported,
   \item[(b)] $\inf_{b\in \rbb_+}\sup_{x\in \rbb}\psi(x)
\E^{-b|x|} < \infty,$
   \item[(c)] $\lim_{n\to \infty} \big(\psi^{(2n)}(0)\big)^{1/n} <
\infty,$
   \item[(d)] $\lim_{x\to\infty}
\displaystyle{\frac{\psi^{\prime}(x)}{\psi(x)}} < \infty.$
   \end{enumerate}
   \end{enumerate}
   \end{thm}
   \begin{proof}
(i) By Lemma~\ref{wkw-supp-zw}, the function $\psi$
extends uniquely to the entire function $\widehat
\psi$ on $\cbb$ given~by
   \begin{align}  \label{nie-ma}
\widehat\psi(z) = \sum_{n=0}^{\infty}
\frac{\gamma_n}{(2n)!} z^{2n}, \quad z\in \cbb,
   \end{align}
where $\gamma_n=\int_{\rbb_+} u^{2n} \D \nu(u) <
\infty$ for all $n\in \zbb_+$. Since $\supp \nu
\neq \{0\},$ $\gamma_n
> 0$ for all $n\in \zbb_+$. Indeed, otherwise
$\gamma_{k} = 0$ for some $k\Ge 1$ (the case
$k=0$ is excluded by the assumption that
$\nu\neq 0$), which implies that $\supp
\nu=\{0\}$, a contradiction. This shows that
$\psi \in \mathscr{H}_{2\bullet}$.

(ii) Set $r= \sup \supp \nu$. We show that
   \begin{align} \label{heliou3}
0 \Le \sup_{x \in (0,\infty)}
\displaystyle{\frac{\psi^{\prime}(x)}{\psi(x)}} \Le r.
   \end{align}
Without loss of generality, we may assume that
$r<\infty$. Clearly, $r= \max \supp \nu$ and $\supp
\nu \subseteq [0,r]$. By Lemma~\ref{wkw-supp-zw}, we
have
   \allowdisplaybreaks
   \begin{align*}
\psi^{\prime}(x) & = \int_{[0,r]} u \sinh(x\, u) \D \nu(u)
   \\
& = \int_{[0,r]} u \tanh(x\, u) \cosh(x\, u) \D \nu(u)
   \\
& \Le r \int_{[0,r]} \cosh(x\, u) \D \nu(u)
   \\
& = r \psi(x), \quad x\in \rbb_+,
   \end{align*}
which yields \eqref{heliou3}.

Set $R=\liminf_{x\to\infty}
\displaystyle{\frac{\psi^{\prime}(x)}{\psi(x)}}$. By
\eqref{cosh-a}, $R \Ge 0$. We claim that
   \begin{align} \label{brzusio}
\supp \nu \subseteq [0,R].
   \end{align}
Without loss of generality, we may assume that $R <
\infty$. Suppose to the contrary that there exits $M
> R$ such that $\nu((M,\infty)) > 0$. Take $L\in
(R,M)$. We show that
   \begin{align}   \label{heliou1}
\lim_{x\to\infty} \frac{\int_{\rbb_+} \cosh(x\, u) \D
\nu(u)}{\int_{(L,\infty)} \cosh(x\, u) \D \nu(u)} = 1.
   \end{align}
For this, it suffices to prove that
   \begin{align} \label{heliou}
\lim_{x\to\infty} \frac{\int_{[0,L]} \cosh(x\, u) \D
\nu(u)}{\int_{(M,\infty)} \cosh(x\, u) \D \nu(u)} = 0.
   \end{align}
That \eqref{heliou} holds may be deduced from the
following estimates:
   \allowdisplaybreaks
   \begin{gather*}
\int_{[0,L]} \cosh(x\, u) \D \nu(u) \Le \nu([0,L])
\cosh(L x), \quad x\in \rbb_+,
   \\
\int_{(M,\infty)} \cosh(x\, u) \D \nu(u) \Ge
\nu((M,\infty)) \cosh(M x), \quad x\in \rbb_+,
   \\
\lim_{x\to \infty} \frac{\cosh(L x)}{\cosh(M x)} =
\lim_{x\to \infty} \frac{\exp(L x)}{\exp(M x)} =0
\quad \text{(because $M > L$).}
   \end{gather*}
A similar reasoning may be used to show that
   \begin{align} \label{heliou2}
\lim_{x\to\infty} \frac{\int_{\rbb_+} u \sinh(x\, u)
\D \nu(u)}{\int_{(L,\infty)} u \sinh(x\, u) \D \nu(u)}
& = 1.
   \end{align}
Combining \eqref{cosh-a} with \eqref{heliou1} and
\eqref{heliou2} yields
   \allowdisplaybreaks
   \begin{align*}
R= \liminf_{x\to\infty}
\displaystyle{\frac{\psi^{\prime}(x)}{\psi(x)}} & =
\liminf_{x\to\infty} \frac{\int_{\rbb_+} u \sinh(x\, u) \D
\nu(u)}{\int_{\rbb_+} \cosh(x\, u) \D \nu(u)}
      \\
& = \liminf_{x\to\infty} \frac{\int_{(L,\infty)} u
\sinh(x\, u) \D \nu(u)}{\int_{(L,\infty)} \cosh(x\, u)
\D \nu(u)}
   \\
& \Ge L \liminf_{x\to\infty} \frac{\int_{(L,\infty)}
\sinh(x\, u) \D \nu(u)}{\int_{(L,\infty)} \cosh(x\, u)
\D \nu(u)}
   \\
& = L \liminf_{x\to\infty} \frac{\int_{(L,\infty)}
\tanh(x\, u) \cosh(x\, u) \D \nu(u)}{\int_{(L,\infty)}
\cosh(x\, u) \D \nu(u)}
   \\
& \Ge L \lim_{x \to \infty} \tanh(x\, L)
   \\
& = L > R,
   \end{align*}
a contradiction. This proves our claim.

It follows that
   \begin{align*}
r \overset{\eqref{brzusio}} \Le R=\liminf_{x\to\infty}
\frac{\psi^{\prime}(x)}{\psi(x)} \Le \limsup_{x\to\infty}
\frac{\psi^{\prime}(x)}{\psi(x)} \Le\sup_{x \in (0,\infty)}
\displaystyle{\frac{\psi^{\prime}(x)}{\psi(x)}}
\overset{\eqref{heliou3}} \Le r,
   \end{align*}
which implies that the limit $\lim_{x\to\infty}
\displaystyle{\frac{\psi^{\prime}(x)}{\psi(x)}}$ exists in
$[0,\infty]$ and \eqref{heliou4-1} holds.

(iii) That the conditions (a) and (d) are equivalent
follows from (ii). By Theorem~\ref{maintheorem}, (a)
implies (b). That (b) implies (a) follows from
Theorems~ \ref{mmainth} and \ref{maintheorem}.
Finally, by Lemma~\ref{wkw-supp-zw} and
Theorem~\ref{maintheorem}, (a) and (c) are equivalent
(cf.\ Remark~\ref{Taylor-gam}). This completes the
proof.
   \end{proof}
Below, we will formulate a version of
Theorem~\ref{limsupq-1} suitable for composition
operators considered in this paper.
   \begin{thm}
Suppose that $\nu$ is a nonzero Borel
measure on $\rbb_+$ such that $\supp
\nu \neq \{0\}$ and
   \begin{align*}
\varphi(x):=\int_{\rbb_+} \cosh(\sqrt{x}\, u) \D
\nu(u) < \infty, \quad x\in \rbb_+.
   \end{align*}
Then $\varphi \in \mathscr{H}_{\bullet}$, the
limit $\lim_{x\to\infty} \sqrt{x} \,
\displaystyle{\frac{\varphi^{\prime}(x)}{\varphi(x)}}$
exists in $[0,\infty]$ and
   \begin{align*}
\lim_{x\to\infty} \sqrt{x} \,
\displaystyle{\frac{\varphi^{\prime}(x)}{\varphi(x)}} =
\sup_{x \in (0,\infty)} \sqrt{x} \,
\displaystyle{\frac{\varphi^{\prime}(x)}{\varphi(x)}} =
\frac 12 \sup \supp \nu.
   \end{align*}
   \end{thm}
   \begin{proof}
Consider the function $\psi\colon \rbb \to \rbb_+$
defined by $\psi(x)= \varphi(x^2)$ for $x \in \rbb.$
Since the hyperbolic cosine $\cosh$ is even, we get
   \begin{align*}
\psi(x)=\int_{\rbb_+} \cosh(x u) \D \nu(u) < \infty,
\quad x\in \rbb.
   \end{align*}
In view of the proof of Theorem~\ref{limsupq-1},
the entire function $\widehat\psi$ on $\cbb$
defined by \eqref{nie-ma} extends $\psi$ and
$\widehat \psi \in \mathscr{H}_{2\bullet}.$
Define the entire function $\widehat\varphi$ on
$\cbb$ by
   \begin{align*}
\widehat\varphi(z) = \sum_{n=0}^{\infty}
\frac{\gamma_n}{(2n)!} z^{n}, \quad z\in \cbb.
   \end{align*}
Clearly $\widehat \varphi \in
\mathscr{H}_{\bullet}.$ Moreover
   \begin{align*}
\widehat\varphi(x) = \widehat\psi (\sqrt{x}) = \psi
(\sqrt{x}) = \varphi (x), \quad x \in \rbb_+,
   \end{align*}
which means that $\widehat\varphi$ is a (necessarily
unique) entire function on $\cbb$ extending $\varphi.$
Applying Theorem~\ref{limsupq-1}(ii), we obtain the
remaining conclusions. This completes the proof.
   \end{proof}
   \section{Applications to composition
operators}
   \subsection{\label{Sec.3.1}Composition operators with affine symbols}
In the subsequent sections we will show that the
cosubnormality of a class of bounded composition
operators on $L^2$ spaces with affine symbols on
the $\kappa$-dimensional Euclidean space is
closely related to the hyperbolic cosine
transform. Let us first establish the necessary
terminology and notation. Given a complex
Hilbert space $\hh$, we denote by $\ogr{\hh}$
the $C^*$-algebra of all bounded linear
operators on~$\hh$. If $S\in \ogr{\hh}$, then
$\jd{S}$ and $\ob{S}$ stand for the kernel and
the range of $S$ respectively. We say that an
operator $S\in \ogr{\hh}$ is
   \begin{enumerate}
   \item[$\bullet$] {\em hyponormal} if $SS^* \Le S^*S$,
   \item[$\bullet$] {\em subnormal} if there exist
a complex Hilbert space $\kk$ and a normal operator
$N\in \ogr{\kk}$ such that $\hh \subseteq \kk$
(isometric embedding) and $Sh=Nh$ for every $h\in
\hh$,
   \item[$\bullet$] {\em cohyponormal} (resp.\ {\em cosubnormal}\/)
if the adjoint $S^*$ of $S$ is hyponormal (resp.\
subnormal).
   \end{enumerate}
It is well known that subnormal operators are
hyponormal, but not conversely. We refer the reader to
\cite{Con91} for the foundations of the theory of
subnormal and hyponormal operators.

To give the reader a better insight
into the topic of this section, we
start with a more general context.
First, note that the only hyponormal
composition operator on an $L^2$ space
over a finite measure space is an
isometric operator (see
\cite[Lemma~7]{Ha-Wh84}). Since the
measure spaces considered in the
forthcoming paragraphs are finite, we
focus our attention on the issue of
cosubnormality. Denote by $V_\kappa$
the $\kappa$-dimensional Lebesgue
measure on $\rbb^\kappa$ ($\kappa \in
\nbb$). Take a Borel function
$\rho\colon \rbb^{\kappa} \to
(0,\infty)$ and define the Borel
measure $\mu_{\rho}$ on $\rbb^{\kappa}$
by
   \begin{align} \label{miuro}
\mu_{\rho}(\varDelta) =
\int_{\varDelta} \rho(x) \D
V_\kappa(x), \quad \varDelta \in
\borel{\rbb^{\kappa}},
   \end{align}
If $\rho(x) = 1$ for every $x\in \rbb^{\kappa}$, we
write $L^2(\rbb^{\kappa})$ in place of
$L^2(\mu_{\rho})$. An affine transformation $T$ of
$\rbb^{\kappa}$ is of the form $T=A+a$, i.e.,
   \begin{align*}
Tx=Ax+a, \quad x\in \rbb^{\kappa},
   \end{align*}
where $A$ is a linear transformation of
$\rbb^\kappa$ and $a\in \rbb^\kappa$. Note that
   \begin{align} \label{iter-ation}
T^n x= A^n x + \sum_{j=0}^{n-1}A^j a, \quad x \in
\rbb^{\kappa}, \, n\in \nbb.
   \end{align}
The identity transformation of $\rbb^\kappa$ is
denoted by $I$. The composition operator
$C_{T,\rho}$ on $L^2(\mu_{\rho})$ is formally
defined by
   \begin{align*}
C_{T,\rho}(f) = f \circ T, \quad f \in
L^2(\mu_{\rho}).
   \end{align*}
The operator $C_{T,\rho}$ is well defined if and only if $A$
is invertible (see, e.g., the proof of
\cite[Theorem~13.1]{jsjs17}), and if this is the case, then
$C_{T,\rho}$ is densely defined and closed.
   \begin{align*}
\text{\em Henceforth, we assume that $A$ is invertible.}
   \end{align*}
This implies that $T$ itself is invertible. The inverse of
$T$, denoted by $S$, is of the form $S=A^{-1} - A^{-1}a$,
i.e.,
   \begin{align*}
Sx=A^{-1}x - A^{-1}a, \quad x\in \rbb^{\kappa}.
   \end{align*}

Now we turn to the class of composition
operators, which we intend to examine in this
section. Let us make the following standing
assumption:
   \begin{align} \label{dens-fin}
   \begin{minipage}{70ex}
$\psi\colon \rbb\to (0,\infty)$ is a continuous
even function and $\rho\colon \rbb^{\kappa} \to
(0,\infty)$ is the density function defined by
   $$ \rho(x)= \psi(\|x\|)^{-1}, \quad
x\in\rbb^{\kappa},
   $$
where $\|x\| = \sqrt{x_1^2 + \ldots +
x_{\kappa}^2}$ for $x=(x_1, \ldots,
x_{\kappa})\in \rbb^{\kappa}$.
   \end{minipage}
   \end{align}
(For simplicity, we do not express explicitly the
dependence of the density $\rho$ on the function
$\psi$ in notation.) Since the measures $V_\kappa$ and
$\mu_{\rho}$ are mutually absolutely continuous, it
follows from \cite[(1.1)]{Da-St97} that if
\eqref{miuro} and \eqref{dens-fin} hold, then
$C_{T,\rho}\in \ogr{L^2(\mu_{\rho})}$ if and only if
$\sup_{x\in \rbb^{\kappa}}
\frac{\psi(\|T(x)\|)}{\psi(\|x\|)} < \infty$, and if
this is the case, then
   \begin{align} \label{nirmy}
\|C_{T,\rho}\|^2 = \frac{1}{|\det A|} \sup_{x\in
\rbb^{\kappa}}\frac{\psi(\|T(x)\|)}{\psi(\|x\|)}.
   \end{align}
Recall that if $C_{T,\rho}\in \ogr{L^2(\mu_{\rho})}$,
then $C_{S,1/\rho} \in \ogr{L^2(\mu_{1/\rho})}$ and
(see \cite[(UE)]{Sto90}; see also
\cite[Remark~$4^{\mathrm o}$, p.\ 403]{Da-St97}):
   \begin{align} \label{ct-sp}
|\det A| C_{T,\rho}^* = U_{\rho}^{-1}
C_{S,1/\rho} U_{\rho},
   \end{align}
where $U_{\rho}\colon L^2(\mu_{\rho}) \to
L^2(\mu_{1/\rho})$ is the unitary isomorphism
defined by $U_{\rho}f=\rho f$ for $f\in
L^2(\mu_{\rho})$. It follows from \eqref{ct-sp}
that $C_{T,\rho}$ is cosubnormal if and only if
$C_{S,1/\rho}$ is subnormal. Using the
change-of-variables theorem (see
\cite[Theorem~7.26]{rud87}) yields
   \begin{align*}
(\mu_{1/\rho} \circ S^{-1})(\Delta) &:=
\mu_{1/\rho} (S^{-1}(\Delta))
   \\
& \hspace{.6ex}= |\det A| \int_{\varDelta}
\frac{\psi(\|T(x)\|)}{\psi(\|x\|)} \D
\mu_{1/\rho}(x), \quad \varDelta \in
\borel{\rbb^{\kappa}},
   \end{align*}
which implies that the Radon-Nikodym derivative
$h_{S,1/\rho} := \D \mu_{1/\rho}\circ S^{-1}/\D
\mu_{1/\rho}$ of $\mu_{1/\rho}\circ S^{-1}$ with
respect to $\mu_{1/\rho}$ is of the form
   \begin{align} \label{exp-sion}
h_{S,1/\rho}(x) = |\det A|
\frac{\psi(\|T(x)\|)}{\psi(\|x\|)}, \quad x \in
\rbb^{\kappa}.
   \end{align}
Using the identity \eqref{iter-ation} and
iterating \eqref{exp-sion}, we see that
   \begin{align*}
h_{S^n,1/\rho}(x) = |\det A|^n
\frac{\psi(\|T^n(x)\|)}{\psi(\|x\|)}, \quad x \in
\rbb^{\kappa}, \, n\in \zbb_+,
   \end{align*}
where $h_{S^n,1/\rho} := \D \mu_{1/\rho}\circ
(S^n)^{-1} / \D \mu_{1/\rho}$. Hence, applying
Lambert's criterion for subnormality (see
\cite[Corollary~4]{Lam88}) to the composition
operator $C_{S,1/\rho}$ and using the fact that
the measures $\mu_{1/\rho}$ and $V_\kappa$ are
mutually absolutely continuous, we get:
   \begin{align} \label{du-du}
   \begin{minipage}{70ex}
{\em $C_{T,\rho} \in \ogr{L^2(\mu_{\rho})}$ is
cosubnormal if and only if
$\{\psi(\|T^n(x)\|)\}_{n=0}^{\infty}$ is a
Stieltjes moment sequence for $V_{\kappa}$-almost
every $x\in \rbb^{\kappa}$.}
   \end{minipage}
   \end{align}
We need the following stronger version of
\eqref{du-du} (cf.\
\cite[Proposition~2.4]{Sto90}).
   \begin{lem} \label{du-du-2}
Suppose that \eqref{miuro} and \eqref{dens-fin} hold.
A composition operator $C_{T,\rho} \in
\ogr{L^2(\mu_{\rho})}$ is cosubnormal if and only if
$\{\psi(\|T^n(x)\|)\}_{n=0}^{\infty}$ is a Stieltjes
moment sequence for every $x\in \rbb^{\kappa}$.
   \end{lem}
   \begin{proof}
In view of \eqref{du-du}, it suffices to prove the
``only if'' part. Assume that $C_{T,\rho}$ is
cosubnormal. First, note that if $n\in \nbb$ and
$\lambdab =\{\lambda_i\}_{i=1}^n \subseteq \cbb$, then
the set
   \begin{align*}
\varOmega_{\lambdab} := \Big\{x\in
\rbb^{\kappa}\colon \sum_{i,j=0}^n
\psi(\|T^{i+j}(x)\|) \lambda_i\bar\lambda_j \Ge
0\Big\}
   \end{align*}
is closed. By the Stieltjes theorem (see
\cite[Theorem~6.2.5]{b-ch-r}) and \eqref{du-du},
we see that $V_{\kappa}(\rbb^{\kappa}\backslash
\varOmega_{\lambdab})=0$. Since the set
$\rbb^{\kappa}\backslash\varOmega_{\lambdab}$ is
open, it must be empty, so
$\varOmega_{\lambdab}=\rbb^{\kappa}$. Hence
$\{x\in \rbb^{\kappa}\colon T(x) \in
\varOmega_{\lambdab}\}=\rbb^{\kappa}$.
Summarizing, we have proved that
   \begin{align*}
\text{$\sum_{i,j=0}^n \psi(\|T^{i+j}(x)\|)
\lambda_i\bar\lambda_j \Ge 0$ and $\sum_{i,j=0}^n
\psi(\|T^{i+j+1}(x)\|) \lambda_i\bar\lambda_j \Ge
0$,}
   \end{align*}
for all finite sequences $\{\lambda_i\}_{i=1}^n
\subseteq \cbb$ and for every $x\in
\rbb^{\kappa}$. Combined with the nontrivial part
of the Stieltjes theorem, this shows that for
every $x\in \rbb^{\kappa}$, the sequence
$\{\psi(\|T^n(x)\|)\}_{n=0}^{\infty}$ is a
Stieltjes moment sequence.
   \end{proof}
The question of the existence of cohyponormal
composition operators $C_{A+a,\rho}$ with a nontrivial
translation part was studied in detail in
\cite[Sections~4 and 5]{Da-St97} (see also
\cite[Theorem~2.2]{Da-St97}). Roughly speaking, the
answer is negative in all but one case, when $\|A\| =
1$, $a \notin \ob{I - AA^*}$ and $\limsup_{x\to\infty}
\psi'(x)/\psi(x) < \infty$. The last inequality is
automatically satisfied if $\psi$ is the hyperbolic
cosine transform of a compactly supported Borel
measure on $\rbb_+$ (see Theorem~\ref{limsupq-1}). If
$A=I$, then trivially $\ob{I - AA^*} = \{0\}$. This
case is analyzed in Theorems~\ref{wkw-uch} and
\ref{blis-zer}, and Proposition~\ref{blis-zur}. If
$\|A\|=1$ and $a \in \jd{I-A} \setminus \{0\}$, then
$a \notin \ob{I - AA^*}$ (see \eqref{trus}). This
issue, in turn, is examined in
Propositions~\ref{impi-ch} and \ref{wkw-ukh}. Finally,
the case when $\|A\| = 1$ and $a \notin \ob{I - AA^*}$
is discussed in Proposition~\ref{wkw-1}
for~$\kappa=1$.
   \subsection{\label{Sec.3.2}Cosubnormality {\em versus} hyperbolic cosine
transform} We begin this section by showing that if a
$C_0$-semigroup $\{C_{I+ta_0,\rho}\}_{t\in \rbb_+}$ of
composition operators induced by translations consists
of cosubnormal operators, where $a_0 \in \rbb^{\kappa}
\backslash \{0\}$, then the function $\psi$ must be
the hyperbolic cosine transform of a compactly
supported Borel measure on $\rbb_+$, and {\em vice
versa}. Before doing so, we will prove the following
general fact. It is worth emphasizing here that
$\lim_{x\to\infty}
\displaystyle{\frac{\psi^{\prime}(x)}{\psi(x)}}$ may
not exist for $\psi \in \mathscr{H}$ (see
\cite{Da96}).
   \begin{lem} \label{ciugl}
Suppose that \eqref{miuro} and \eqref{dens-fin} hold
and either $\psi=\mathrm{const}$ or $\psi$ is a member
of $\mathscr{H}$ such that $\limsup_{x\to\infty}
\displaystyle{\frac{\psi^{\prime}(x)}{\psi(x)}} <
\infty$. Then $C_{I+a,\rho} \in \ogr{L^2(\mu_{\rho})}$
for every $a\in \rbb^{\kappa}$, and the map
$\rbb^{\kappa} \ni a \longmapsto C_{I+a,\rho}\in
\ogr{L^2(\mu_{\rho})}$ is continuous in the strong
operator topology.
   \end{lem}
   \begin{proof}
Set $\varDelta_{\theta}=\{x \in \rbb^{\kappa}\colon
\|x\| \Le \theta\}$ for $\theta>0$. Since $\psi$ is
monotonically increasing on $\rbb_+$ and
$\limsup_{x\to\infty}
\displaystyle{\frac{\psi^{\prime}(x)}{\psi(x)}} <
\infty$, we infer from \cite[Lemma~1.1(ii)]{Da-St97}
with $\varphi=\psi$ and $\tau=0$ that for every
$\theta >0$,
   \begin{align*}
\sup_{a \in \varDelta_{\theta}} \sup_{x\in
\rbb^{\kappa}} \frac{\psi(\|x+a\|)}{\psi(\|x\|)} \Le
\sup_{a \in \varDelta_{\theta}} \sup_{x\in
\rbb^{\kappa}} \frac{\psi(\|x\|+\|a\|)}{\psi(\|x\|)}
\Le \sup_{x\in \rbb^{\kappa}}
\frac{\psi(\|x\|+\theta)}{\psi(\|x\|)} < \infty.
   \end{align*}
This together with \eqref{nirmy} implies that
$C_{I+a,\rho} \in \ogr{L^2(\mu_{\rho})}$ for every
$a\in \rbb^{\kappa}$, and
   \begin{align} \label{kdwq}
\sup_{a \in \varDelta_{\theta}} \|C_{I+a, \rho}\| <
\infty, \quad \theta >0.
   \end{align}

Next, we show that $\lim_{h\to 0} C_{I+b+h, \rho}f =
C_{I+b, \rho}f$ for all $f\in L^2(\mu_{\rho})$ and
$b\in \rbb^{\kappa}$. Fix $b \in \rbb^{\kappa}$. Let
$f\colon \rbb^{\kappa} \to \cbb$ be a continuous
function with compact support. Let $r>0$ be such that
$\varDelta_r$ contains the support of $f$. Set $R = r
+ \|b\| + 1$. Observe that $f$ is uniformly
continuous. Therefore, for $\varepsilon> 0$, there
exists $\delta \in (0,1)$ such that
   \begin{align} \label{jak-num-0}
\forall x,y \in \rbb^{\kappa}\colon \|x-y\| \Le \delta
\implies |f(x)-f(y)| \Le \varepsilon.
   \end{align}
Then
   \begin{align} \label{jak-num}
\forall x \in \rbb^{\kappa}\backslash \varDelta_R \;
\; \forall h\in \varDelta_{\delta}\colon f(x+b+h)=0.
   \end{align}
It follows that \allowdisplaybreaks
   \begin{align*}
\|C_{I+b+h, \rho}f - C_{I+b, \rho}f\|^2 & =
\int_{\rbb^{\kappa}} |f(x+b+h) - f(x+b)|^2
\frac{1}{\psi(\|x\|)}\D V_{\kappa} (x)
   \\
&\hspace{1ex}
\hspace{-1.7ex}\overset{\eqref{jak-num}}=
\int_{\varDelta_R} |f(x+b+h) - f(x+b)|^2
\frac{1}{\psi(\|x\|)} \D V_{\kappa} (x)
   \\
& \hspace{-.9ex} \overset{\eqref{jak-num-0}} \Le
\varepsilon^2 \int_{\varDelta_R} \frac{1}{\psi(\|x\|)}
\D V_{\kappa} (x) , \quad h\in \varDelta_{\delta},
   \end{align*}
so $\lim_{h\to 0} C_{I+b+h, \rho}f = C_{I+b, \rho}f$
for every continuous function $f\colon \rbb^{\kappa}
\to \cbb$ with compact support. Using \eqref{kdwq} and
the fact that continuous complex functions on
$\rbb^{\kappa}$ with compact support are dense in
$L^2(\mu_{\rho})$ (see \cite[Theorem~3.14]{rud87}), we
conclude that the map $\rbb^{\kappa} \ni a \longmapsto
C_{I+a,\rho}\in \ogr{L^2(\mu_{\rho})}$ is continuous
in the strong operator topology. This completes the
proof.
   \end{proof}
We can now prove the main result of this paper.
   \begin{thm} \label{wkw-uch}
Suppose that \eqref{miuro} and \eqref{dens-fin} hold
and $a_0 \in \rbb^{\kappa} \backslash \{0\}$. Then the
following conditions are equivalent{\em :}
   \begin{enumerate}
   \item[(i)] $C_{I+ta_0,\rho}$ is a bounded cosubnormal
operator on $L^2(\mu_{\rho})$ for every $t\in \rbb_+$ and
$\lim_{t \to 0+} C_{I+ta_0, \rho}=C_{I,\rho}$ in the strong
operator topology,
   \item[(ii)] $\psi$ is the hyperbolic cosine transform of a
compactly supported Borel measure on $\rbb_+$.
   \end{enumerate}
Moreover, if {\em (ii)} holds, then $C_{I+a,\rho} \in
\ogr{L^2(\mu_{\rho})}$ and $C_{I+a,\rho}$ is
cosubnormal for every $a\in \rbb^{\kappa}$, and the
map $\rbb^{\kappa} \ni a \longmapsto C_{I+a,\rho}\in
\ogr{L^2(\mu_{\rho})}$ is continuous in the strong
operator topology.
   \end{thm}
   \begin{proof}
(i)$\Rightarrow$(ii) Set $\varLambda_t =
C_{I+ta_0,\rho}$ for $t\in \rbb_+$. It is easily seen
that $\varLambda_{0}$ is the identity operator on
$L^2(\mu_{\rho})$ and $\varLambda_{s+t} =
\varLambda_{s}\varLambda_{t}$ for all $s,t \in
\rbb_+$. Hence, by the continuity hypothesis,
$\{\varLambda_t\}_{t\in \rbb_+}$ is a $C_0$-semigroup.
It follows from \cite[Theorem~1.2.2]{Paz83} that there
exist $\omega \in \rbb_+$ and $M\in [1,\infty)$ such
that
   \begin{align} \label{num-er}
\|\varLambda_t\|^2 \Le M \E^{\omega t}, \quad t
\in \rbb_+.
   \end{align}
By \eqref{nirmy}, we have
   \begin{align*}
\|\varLambda_t\|^2 = \sup_{x\in \rbb^{\kappa}}
\frac{\psi(\|x+ta_0\|)}{\psi(\|x\|)}, \quad t \in
\rbb_+.
   \end{align*}
This together with \eqref{num-er} implies that
   \begin{align} \label{sia-su}
\psi(\|x+ta_0\|) \Le M \psi(\|x\|) \E^{\omega
t}, \quad x \in \rbb^{\kappa}, \, t \in \rbb_+.
   \end{align}
Substituting $x=0$ and $t=\frac{s}{\|a_0\|}$
into \eqref{sia-su}, we see that
   \begin{align*}
\psi(s) \Le M_1 \E^{\omega_1 s}, \quad s\in
\rbb_+.
   \end{align*}
where $M_1 = M \psi(0)$ and
$\omega_1=\frac{\omega}{\|a_0\|}$. Since $\psi$
is even, we have
   \begin{align} \label{psim}
\psi(s) \Le M_1 \E^{\omega_1 |s|}, \quad s\in
\rbb.
   \end{align}
It follows from \eqref{iter-ation} and
Lemma~\ref{du-du-2} applied to $T=I+ta_0$ that the
sequence $\{\psi(\|x + n t a_0\|)\}_{n=0}^{\infty}$ is
a Stieltjes moment sequence for all $x\in
\rbb^{\kappa}$ and $t\in \rbb_+$. In particular,
substituting $x=\frac{\xi}{\|a_0\|}a_0$ and $t=
\frac{s}{\|a_0\|}$ and using the assumption that
$\psi$ is even, we deduce that for all $\xi\in \rbb$
and $s\in \rbb_+$, the sequence $\{\psi(\xi + n
s)\}_{n=0}^{\infty}$ is a Stieltjes moment sequence,
which as such is positive definite on the
$*$-semigroup $(\zbb_+,+,n^*=n)$. Applying
Lemma~\ref{kiedy-cosub-cosh}, w conclude that $\psi$
is exponentially convex. Hence, by \eqref{psim} and
Theorem~\ref{maintheorem} , $\psi$ is the hyperbolic
cosine transform of a compactly supported Borel
measure on $\rbb_+$.

(ii)$\Rightarrow$(i) Suppose that $\psi$ is the hyperbolic
cosine transform of a compactly supported Borel measure $\nu$
on $\rbb_+$. Since $\psi\neq 0$, $\nu\neq 0$. If $\supp{\nu}
= \{0\}$, then $\psi=\mathrm{const}$. In turn, if $\supp{\nu}
\neq \{0\}$, then, by Theorem~\ref{limsupq-1}(i), $\psi\in
\mathscr{H}$. It follows from Theorem~\ref{limsupq-1}(iii)
and Lemma~\ref{ciugl} that $C_{I+a,\rho} \in
\ogr{L^2(\mu_{\rho})}$ for every $a\in \rbb^{\kappa}$, and
the map $\rbb^{\kappa} \ni a \longmapsto C_{I+a,\rho}\in
\ogr{L^2(\mu_{\rho})}$ is continuous in the strong operator
topology. In particular, $\{C_{I+ta_0,\rho}\}_{t\in \rbb_+}$
is a $C_0$-semigroup.

Next, we will show that $C_{I+a, \rho}$ is cosubnormal
for every $a\in \rbb^{\kappa}$. Fix $a\in
\rbb^{\kappa} \backslash \{0\}$ (the case $a=0$ is
trivial). Let $\varphi_{\xi\eta}$ be as in
Theorem~\ref{Katz}. Denote by $P$ the orthogonal
projection of $\rbb^{\kappa}$ onto $\langle a\rangle$,
the one-dimensional span of $\{a\}$. Let $P^{\perp}$
stand for the orthogonal projection of $\rbb^{\kappa}$
onto $\langle a\rangle^{\perp}$. Clearly,
$P^{\perp}=I-P$ and $Px = s_xa$ for $x\in
\rbb^{\kappa}$, where $s_x =
\frac{\is{x}{a}}{\|a\|^2}$. Fix $u\in \rbb_+$, $x\in
\rbb^{\kappa}$ and $\varepsilon \in \{0,1\}$. Then
   \begin{align*}
\cosh(u\|x + n a\|) & =
\cosh\Big(\sqrt{\|ua\|^2(n+s_x)^2 +
\|uP^{\perp}x\|^2}\Big)
   \\
&= \varphi_{\xi\eta}(n+s_x), \quad n\in \zbb_+,
   \end{align*}
where $\xi=\|ua\|^2$ and $\eta=\|uP^{\perp}x\|^2$.
Hence, by Theorem~\ref{Katz}, we have
   \begin{multline*}
\sum_{i,j=1}^{m} \cosh(u\|x+(n_i+n_j+\varepsilon)a\|)
\lambda_i\bar\lambda_j
   \\
= \sum_{i,j=1}^{m} \varphi_{\xi\eta}\Big(\Big(n_i +
\frac{s_x+\varepsilon}{2}\Big)+ \Big(n_j +
\frac{s_x+\varepsilon}{2}\Big)\Big)
\lambda_i\bar\lambda_j \Ge 0,
   \end{multline*}
for all finite sequences
$\{\lambda_j\}_{j=1}^{m}\subseteq \cbb$ and
$\{n_j\}_{j=1}^{m} \subseteq \zbb_+$. Thus, the
sequence $\{\cosh(u\|x+(n+\varepsilon)
a\|)\}_{n=0}^{\infty}$ is positive definite on the
$*$-semigroup $(\zbb_+,+,n^*=n)$ for $\varepsilon=0,1$
and $u\in \rbb_+$. In view of \eqref{ch-tr-1}, the
sequence $\{\psi(\|x+(n+\varepsilon)
a\|)\}_{n=0}^{\infty}$ is positive definite on the
$*$-semigroup $(\zbb_+,+,n^*=n)$ for
$\varepsilon=0,1$. By the Stieltjes theorem (see
\cite[Theorem~6.2.5]{b-ch-r}), the sequence
$\{\psi(\|x+n a\|)\}_{n=0}^{\infty}$ is a Stieltjes
moment sequence for every $x\in \rbb^{\kappa}$.
Combined with Lemma~\ref{du-du-2}, this implies that
$C_{I+a, \rho}$ is cosubnormal. This completes the
proof.
   \end{proof}
The following is a direct consequence of Lemma~
\ref{ciugl} and Theorem~\ref{wkw-uch}.
   \begin{cor} \label{winej}
Suppose that \eqref{miuro} and \eqref{dens-fin} hold
and either $\psi=\mathrm{const}$ or $\psi$ is a member
of $\mathscr{H}$ such that $\limsup_{x\to\infty}
\displaystyle{\frac{\psi^{\prime}(x)}{\psi(x)}} <
\infty$. Then $C_{I+a,\rho} \in \ogr{L^2(\mu_{\rho})}$
for every $a\in \rbb^{\kappa}$, and the following
conditions are equivalent{\em :}
   \begin{enumerate}
   \item[(i)] $C_{I+a,\rho}$ is cosubnormal for every $a\in
\rbb^{\kappa}$,
   \item[(ii)] $\psi$ is the hyperbolic cosine transform of a
compactly supported Borel measure on $\rbb_+$.
   \end{enumerate}
   \end{cor}
The implication (i)$\Rightarrow$(ii) of
Theorem~\ref{wkw-uch} can be generalized as follows
(see also Remark~ \ref{cos-to}).
   \begin{pro} \label{impi-ch}
Suppose that \eqref{miuro} and \eqref{dens-fin} hold. Assume
also that $a_0 \in \jd{I - A} \backslash \{0\}$,
$C_{A+ta_0,\rho}$ is a bounded cosubnormal operator on
$L^2(\mu_{\rho})$ for every $t\in \rbb_+$ and there exists
$\varepsilon \in (0,\infty)$ such that
   \begin{align}  \label{srup}
\sup_{0 \Le t \Le \varepsilon}
\|C_{A+ta_0,\rho}\| < \infty.
   \end{align}
Then $\psi$ is the hyperbolic cosine transform
of a compactly supported Borel measure on
$\rbb_+$.
   \end{pro}
   \begin{proof}
Since the space $\jd{I-A}$ is unitarily
equivalent to $\rbb^{\kappa_1}$ for some $1\Le
\kappa_1 \Le \kappa$, we can consider the
composition operator $C_{I_1+ta_0, \rho_1}$ on
$L^2(\mu_{\rho_1})$, where $I_1$ is the identity
transformation of $\jd{I-A}$,
$\rho_1=\rho|_{\jd{I-A}}$ and $\mu_{\rho_1}$ is
the Borel measure defined by \eqref{miuro} with
$\rbb^{\kappa_1}$ in place of $\rbb^{\kappa}$.

We divide the proof into three steps.

{\sc Step 1.} $C_{I_1+ta_0, \rho_1} \in
\ogr{L^2(\mu_{\rho_1})}$ for every $t \in
\rbb_+$.

Indeed, by \eqref{nirmy}, we have
   \begin{align*}
\|C_{A+ta_0, \rho}\|^2 = \frac{1}{|\det A|}
\sup_{x\in \rbb^{\kappa}}
\frac{\psi(\|Ax+ta_0\|)}{\psi(\|x\|)}, \quad t
\in \rbb_+.
   \end{align*}
This implies that
   \begin{align} \label{ogrod}
\sup_{x\in \jd{I-A}}
\frac{\psi(\|(I_1+ta_0)(x)\|)}{\psi(\|x\|)} \Le
|\det A| \|C_{A+ta_0, \rho}\|^2 < \infty, \quad
t \in \rbb_+.
   \end{align}
Hence, by \eqref{nirmy}, $C_{I_1+ta_0, \rho_1} \in
\ogr{L^2(\mu_{\rho_1})}$ for every $t \in \rbb_+$.

{\sc Step 2.} $C_{I_1+ta_0, \rho_1}$ is
cosubnormal for every $t \in \rbb_+$.

Indeed, by Lemma~\ref{du-du-2} applied to
$C_{A+ta_0,\rho}$,
$\{\psi(\|(A+ta_0)^n(x)\|)\}_{n=0}^{\infty}$ is
a Stieltjes moment sequence for every $x\in
\rbb^{\kappa}$. Therefore,
$\{\psi(\|(I_1+ta_0)^n(x)\|)\}_{n=0}^{\infty}$
is a Stieltjes moment sequence for every $x\in
\jd{I-A}$. Applying Step 1 and
Lemma~\ref{du-du-2}, we conclude that
$C_{I_1+ta_0, \rho_1}$ is cosubnormal for every
$t \in \rbb_+$.

{\sc Step 3.} $\lim_{t \to 0+} C_{I_1+ta_0,
\rho_1}=C_{I_1,\rho_1}$ in the strong operator topology.

Indeed, arguing as in the proof of Lemma~\ref{ciugl},
we see that
   \begin{align*}
\lim_{h \to 0+} C_{I_1+ha_0, \rho_1}f = C_{I_1,
\rho_1}f
   \end{align*}
for every continuous function $f\colon \jd{I-A} \to \cbb$
with compact support. It follows from \eqref{nirmy},
\eqref{srup} and \eqref{ogrod} that $\sup_{0 \Le t \Le
\varepsilon} \|C_{I_1+ta_0,\rho_1}\| < \infty$. This together
with the fact that continuous complex functions on $\jd{I-A}$
with compact support are dense in $L^2(\mu_{\rho_1})$ implies
that $\lim_{t \to 0+} C_{I_1+ta_0, \rho_1}=C_{I_1,\rho_1}$ in
the strong operator topology.

Finally, using Steps 1, 2 and 3, we can apply
implication (i)$\Rightarrow$(ii) of
Theorem~\ref{wkw-uch} to $\{C_{I_1+ta_0,
\rho_1}\}_{t\in \rbb_+}$. This completes the proof.
   \end{proof}
   \begin{rem} \label{cos-to}
Regarding the condition (i) of Theorem~\ref{wkw-uch}, it is
worth noting that $\lim_{t \to 0+} C_{I+ta_0,
\rho}=C_{I,\rho}$ in the strong operator topology if and only
if $\sup_{0\Le t \Le \varepsilon} \|C_{I+ta_0, \rho}\| <
\infty$ for some $\varepsilon \in (0,\infty)$ (cf.\
\eqref{srup}). Indeed, the ``only if'' part follows from
\eqref{num-er}. In turn, the ``if'' part is a direct
consequence of Step~3 of the proof of
Proposition~\ref{impi-ch} applied to $A=I$.
   \hfill{$\diamondsuit$}
   \end{rem}
The implication (i)$\Rightarrow$(ii) of
Theorem~\ref{wkw-uch} can also take the form as in
Proposition~\ref{wkw-ukh} below. Given a sequence
$\{a_k\}_{k=1}^{\infty} \subseteq \rbb^{\kappa}
\backslash \{0\}$, we write $a_k \rightsquigarrow 0$
if $\frac{1}{k \|a_k\|} \in \nbb$ for every $k\in
\nbb$. Clearly, $a_k \rightsquigarrow 0$ implies $a_k
\to 0$, but not conversely.
   \begin{pro} \label{wkw-ukh}
Suppose that \eqref{miuro} and \eqref{dens-fin}
hold, $\psi$ is a member of $\mathscr{H}$ such
that $\limsup_{x\to\infty}
\frac{\psi^{\prime}(x)}{\psi(x)} < \infty$ and
$\|A\|=1$. Then for every $a\in \jd{I - A}$,
$C_{A+a,\rho} \in \ogr{L^2(\mu_{\rho})}$.
Moreover, if one of the following conditions is
valid{\em :}
   \begin{enumerate}
   \item[(i)] $C_{A+a_k,\rho}$ is cosubnormal for every $k\in
\nbb$ and for some sequence $\{a_k\}_{k=1}^{\infty}
\subseteq \jd{I - A}\backslash \{0\}$ such that $a_k
\rightsquigarrow 0$,
   \item[(ii)] $C_{A+t a_0,\rho}$ is cosubnormal for every $t\in
\rbb$ such that $0 < |t| < \varepsilon$ and for some
$\varepsilon >0$ and $a_0 \in \jd{I - A}\backslash \{0\}$,
   \end{enumerate}
then $\psi$ is the hyperbolic cosine transform
of a compactly supported Borel measure on
$\rbb_+$.
   \end{pro}
   \begin{proof}
Since $\psi$ is even and has a power series
expansion at the origin, we deduce that
$\psi^{(n)}(0)=0$ for $n=1,3,5, \ldots$. This
together with \eqref{dens-fin} implies that the
function $\varphi\colon \rbb_+ \to \rbb$ defined
by $\varphi(x)=\psi(\sqrt{x})$ for $x\in \rbb_+$
is a member of $\mathscr{H}_0$ such that
$\limsup_{x\to\infty} \sqrt{x} \,
\frac{\varphi^{\prime}(x)}{\varphi(x)} < \infty$
(see the paragraph below \eqref{dub-sta}).

It follows from the equality $\|A\|=1$ that
$\jd{I - A} = \jd{I - A^*}$ (see
\cite[Proposition~I.3.1]{Sz-F70}). This implies
that
   \begin{align*}
\jd{I - A} \subseteq \jd{I - AA^*} = \ob{I -
AA^*}^{\perp},
   \end{align*}
which is equivalent to
   \begin{align} \label{trus}
\jd{I - A}\perp \ob{I - AA^*}.
   \end{align}
Hence, if $a \in \jd{I - A}$, then by
\cite[Corollary~2.4]{Da-St97}, $C_{A+a,\rho} \in
\ogr{L^2(\mu_{\rho})}$.

(i) Assume now that $C_{A+a_k,\rho}$ is cosubnormal
for every $k\in \nbb$, where $\{a_k\}_{k=1}^{\infty}
\subseteq \jd{I - A}\backslash \{0\}$ is such that
$a_k \rightsquigarrow 0$, that is, $s_k:=\frac{1}{k
\|a_k\|} \in \nbb$ for every $k\in \nbb$. We infer
from Lemma~\ref{du-du-2} that $\{\psi(\|(A+
a_k)^n(x)\|)\}_{n=0}^{\infty}$ is a Stieltjes moment
sequence for all $x\in \rbb^{\kappa}$ and $k\in \nbb$.
By assumption and \eqref{iter-ation}, we have
   \begin{align*}
(A+a_k)^n(x) = x + na_k, \quad x \in \jd{I - A}, \, k
\in \nbb, \, n \in \zbb_+.
   \end{align*}
This implies that
   \begin{align*}
\|(A+ a_k)^n(tks_k a_k)\| = \Big|t + \frac{n}{k\cdot
s_k}\Big|, \quad t \in \rbb, \, k\in \nbb, \, n\in
\zbb_+.
   \end{align*}
Since the function $\psi$ is even, we see that for all $t\in
\rbb$ and $k\in \nbb$, $\{\psi(t + \frac{n}{k\cdot
s_k})\}_{n=0}^{\infty}$ is a Stieltjes moment sequence, which
as such is positive definite on the $*$-semigroup
$(\zbb_+,+,n^*=n)$. Applying
Lemma~\ref{kiedy-cosub-cosh}(iv), we deduce that $\psi$ is
exponentially convex. From the assumption that $\psi$ is even
and from the ``moreover'' part of \cite[Lemma~1.1]{Da-St97}
with $\varphi=\psi$ and $\sigma=\tau=0$, it follows that
   \begin{align*}
\psi(t) = \psi(|t|) \Le \psi(0) \E^{b|t|}, \quad
t \in \rbb,
   \end{align*}
where $b:=\sup_{t \in \rbb_+}
\frac{\psi^{\prime}(t)}{\psi(t)}$. By
Theorem~\ref{maintheorem}, $\psi$ is the
hyperbolic cosine transform of a compactly
supported Borel measure on $\rbb_+$.

(ii) This case can be easily deduced from (i).
   \end{proof}
By Theorem~\ref{wkw-uch}, the cosubnormality of the
$C_0$-semigroup $\{C_{I+ta_0,\rho}\}_{t\in \rbb_+}$
with $a_0 \in \rbb^{\kappa} \backslash \{0\}$ implies
that $\psi$ is the hyperbolic cosine transform of a
compactly supported Borel measure on $\rbb_+$. On the
other hand, as shown in Proposition~\ref{wkw-1} below,
if $\kappa=1$ and $\psi$ is the hyperbolic cosine
transform of a compactly supported Borel measure $\nu$
on $\rbb_+$ with $\supp \nu \neq \{0\}$, the
cosubnormality of $C_{A+a,\rho}$ forces $A$ to be
equal to $I$.
   \begin{pro} \label{wkw-1}
Suppose that \eqref{miuro} and \eqref{dens-fin}
hold, $\psi$ is the hyperbolic cosine transform
of a compactly supported Borel measure $\nu$ on
$\rbb_+$ and
   \begin{align} \label{st-ass}
   \begin{minipage}{60ex}
$T=A+a$, $\|A\|=1$ and $a \not \in \ob{I -
AA^*}$.
   \end{minipage}
   \end{align}
Then $C_{T,\rho} \in \ogr{L^2(\mu_{\rho})}$ and
the following statements hold{\em :}
   \begin{enumerate}
   \item[(i)] if $\supp{\nu} = \{0\}$, then $|\det
A|^{1/2} C_{T,\rho}$ is a unitary operator,
   \item[(ii)] if $A=I$, then
$C_{T,\rho}$ is cosubnormal,
   \item[(iii)] if  $\kappa=1$, $\supp \nu \neq
\{0\}$ and $C_{T,\rho}$ is cosubnormal, then $A=I$.
   \end{enumerate}
   \end{pro}
   \begin{proof}
Consider the case where $\supp{\nu} \neq \{0\}$. Then $\psi
\in \mathscr{H}_{2\bullet}$ (see Theorem~\ref{limsupq-1}(i)).
Using \eqref{st-ass}, Theorem~\ref{limsupq-1}(iii) and
\cite[Corollary~2.4(iii)]{Da-St97} with $\varphi$ as in the
proof of Proposition~\ref{wkw-ukh}, we deduce that
$C_{T,\rho} \in \ogr{L^2(\mu_{\rho})}$.

(i) Suppose that $\supp{\nu} = \{0\}$. Straightforward
computations show that if $B$ is an invertible linear
transformation of $\rbb^\kappa$ and $b\in \rbb^{\kappa}$,
then $|\det B|^{1/2} C_{B+b,\rho}$ is an isometry. This when
applied to $T^{-1}$ gives $C_{T^{-1},\rho} \in
\ogr{L^2(\mu_{\rho})}$, which implies that $C_{T,\rho}
C_{T^{-1},\rho} = C_{I,\rho}$. Hence, $|\det A|^{1/2}
C_{T,\rho}$ is surjective, and so $|\det A|^{1/2} C_{T,\rho}$
is a unitary operator.

(ii) This statement is a direct consequence of
Theorem~\ref{wkw-uch}.

(iii) Assume that $\kappa=1$, $\supp \nu \neq \{0\}$ and
$C_{T,\rho}$ is cosubnormal. Since $\|A\|=1$, we have two
possibilities: either $A=-I$ or $A=I$. The first possibility
is excluded by \cite[Theorem~4.4(ii)]{Da-St97} with $\varphi$
as in the proof of Proposition~\ref{wkw-ukh}. Hence we have
$A=I$.
   \end{proof}
   \subsection{\label{Sec.3.3}Lack of cosubnormality}
It follows from Corollary~\ref{winej} that if
$\psi\in\mathscr{H} \cup \{\mathrm{const}\}$ satisfies
the assumptions of this corollary, then $\psi$ is not
the hyperbolic cosine transform of a compactly
supported Borel measure on $\rbb_+$ if and only if
there exists $a\in \rbb^{\kappa}\backslash \{0\}$ such
that $C_{I+a,\rho}$ is not cosubnormal. In fact, as
shown below, the set of all $a\in \rbb^{\kappa}$ for
which $C_{I+a,\rho}$ is not cosubnormal is open.
   \begin{lem} \label{locyl}
Suppose that \eqref{miuro} and \eqref{dens-fin} hold
and for every $a\in \rbb^{\kappa}$, $C_{I+a,\rho} \in
\ogr{L^2(\mu_{\rho})}$. Then the set
   \begin{align*}
\varXi_{\rho}:=\Big\{a\in \rbb^{\kappa}\colon
C_{I+a,\rho} \text{ is not cosubnormal}\Big\}
   \end{align*}
is open.
   \end{lem}
   \begin{proof}
It is enough to show that the complement of
$\varXi_{\rho}$ is closed. Let $a_0\in \rbb^{\kappa}$.
Take a sequence $\{a_k\}_{k=1}^{\infty} \subseteq
\rbb^{\kappa}$ such that $C_{I+a_k,\rho}$ is
cosubnormal for every $k\in \nbb$, and $\lim_{k\to
\infty} a_k=a_0$. By Lemma~\ref{du-du-2}, the sequence
$\{\psi(\|x+na_k\|)\}_{n=0}^{\infty}$ is a Stieltjes
moment sequence for every $x\in \rbb^{\kappa}$ and
every $k \in \nbb$. Since the class of Stieltjes
moment sequences is closed in the topology of
pointwise convergence (see
\cite[Theorem~6.2.5]{b-ch-r}) and
   \begin{align*}
\lim_{k\to\infty} \psi(\|x+na_k\|) = \psi(\|x+na_0\|),
\quad x \in \rbb^{\kappa}, \, n\in \zbb_+,
   \end{align*}
we see that the sequence
$\{\psi(\|x+na_0\|)\}_{n=0}^{\infty}$ is a Stieltjes
moment sequence for every $x\in \rbb^{\kappa}$. Using
Lemma~\ref{du-du-2} again, we conclude that
$C_{I+a_0,\rho}$ is cosubnormal.
   \end{proof}
For certain functions $\psi$, the set $\varXi_{\rho}$ may
contain a punctured disk centered at 0. The following theorem
does not seem to follow from Lemma~\ref{locyl}.
   \begin{thm} \label{blis-zer}
Suppose that \eqref{miuro} and \eqref{dens-fin} hold,
$k\in \{4,6,8, \ldots\}$ and $\psi$ is a member of
$\mathscr{H}_k$ such that $\limsup_{x\to\infty}
\frac{\psi^{\prime}(x)}{\psi(x)} < \infty$. Then
$C_{I+a,\rho} \in \ogr{L^2(\mu_{\rho})}$ for every
$a\in \rbb^{\kappa}$ and there exists $\varepsilon \in
(0,\infty)$ such that for every $a\in \rbb^{\kappa}$
with $0 < \|a\| < \varepsilon$, the operator
$C_{I+a,\rho}$ is not cosubnormal.
   \end{thm}
   \begin{proof}
It follows from Lemma~\ref{ciugl} that $C_{I+a,\rho}
\in \ogr{L^2(\mu_{\rho})}$ for every $a\in
\rbb^{\kappa}$. It is easy to verify that the function
$\varphi\colon \rbb_+ \to \rbb$ defined by
$\varphi(x)=\psi(\sqrt{x})$ for $x\in \rbb_+$ is a
member of $\mathscr{H}_s$ with $s:=\frac{1}{2} k$
(cf.\ the proof of Proposition~\ref{wkw-ukh}). Suppose
to the contrary that for every $j\in \nbb$, there
exists $a_j \in \rbb^{\kappa}$ such that $0 < \|a_j\|
< \frac{1}{j}$ and $C_{I+a_j,\rho}$ is cosubnormal. By
Lemma~\ref{du-du-2}, \eqref{iter-ation} and
\cite[Theorem~6.2.5]{b-ch-r}, the sequence
$\{\varphi(\|x + n a_j\|^2)\}_{n=0}^{\infty}$ is
positive definite on the $*$-semigroup
$(\zbb_+,+,n^*=n)$ for all $x\in \rbb^{\kappa}$ and
$j\in \nbb$. Substituting $x=0$, we conclude that the
sequence $\{\varphi(n^2\|a_j\|^2)\}_{n=0}^{\infty}$ is
positive definite on the $*$-semigroup
$(\zbb_+,+,n^*=n)$ for every $j\in \nbb$. Since
$\varphi \in \mathscr{H}_s$ with $s \Ge 2$ and
$\lim_{j\to \infty} \|a_j\|=0$, we deduce from the
proof of \cite[Lemma~5.2]{Sto91} that
   \begin{align} \label{cpds}
   \begin{minipage}{60ex}
the sequence $\{n^{2s}\}_{n=0}^{\infty}$ is
conditionally positive definite on the
$*$-semigroup $(\zbb_+,+,n^*=n)$.
   \end{minipage}
   \end{align}
It follows from
\cite[Proposition~2.2.11]{Ja-Ju-St20} that if
$j\in \zbb_+$, then the sequence
$\{n^{j}\}_{n=0}^{\infty}$ is conditionally
positive definite on the $*$-semigroup
$(\zbb_+,+,n^*=n)$ if and only if $j\Le 2$.
Combined with \eqref{cpds}, this implies that $s
\Le 1$, which contradicts the assumption that
$s=\frac{1}{2} k \Ge 2$. This completes the
proof.
   \end{proof}
The particular choice of the function $\psi$ in
Theorem~\ref{blis-zer} shows that all the operators
$\{C_{I+a,\rho}\colon a \in \rbb^{\kappa}\}$ can be
cohyponormal.
   \begin{pro}\label{blis-zur}
Let $\delta$ be a fixed real number such that
$\delta \Ge \frac{2 \cosh \pi}{\cosh \pi - 1}$
and let $\psi \colon \rbb \to \rbb$ be the
function defined by
   \begin{align*}
\psi(x) = \cosh(x) + \cos(x) + \delta, \quad x
\in \rbb.
   \end{align*}
Then the following statements hold{\em :}
   \begin{enumerate}
   \item[(i)] $\psi \colon \rbb\to (0,\infty)$ is a continuous even
function such that $\psi \in \mathscr{H}_{4}$,
   \item[(ii)] $C_{I+a,\rho} \in \ogr{L^2(\mu_{\rho})}$ and
$C_{I+a,\rho}$ is cohyponormal for every $a \in
\rbb^{\kappa}$, where $\mu_{\rho}$ is as in {\em
\eqref{miuro}} and {\em \eqref{dens-fin}},
   \item[(iii)] there exists $\varepsilon \in (0,\infty)$ such that
for every $a\in \rbb^{\kappa}$ with $0 < \|a\| <
\varepsilon$, the operator $C_{I+a,\rho}$ is not
cosubnormal.
   \end{enumerate}
   \end{pro}
   \begin{proof}
Clearly, $\psi$ is a continuous even function such
that $\psi(\rbb) \subseteq (0,\infty)$, $\psi \in
\mathscr{H}_{4}$ and $\limsup_{x\to\infty}
\frac{\psi^{\prime}(x)}{\psi(x)} < \infty$, so by
Lemma~\ref{ciugl}, $C_{I+a,\rho} \in
\ogr{L^2(\mu_{\rho})}$ for every $a \in
\rbb^{\kappa}$.

We claim that the function $\log \psi$ is convex. It
suffices to show that $(\log\psi)^{\prime\prime} \Ge
0$, or equivalently that
   \begin{align*}
\varTheta:=\psi^{\prime\prime} \psi -
\psi^{\prime 2} \Ge 0.
   \end{align*}
Straightforward computations yield
   \begin{align*}
\varTheta(x) = \cosh(x) \Big(2 \sin(x) \tanh(x)
+ \delta \Big(1 -
\frac{\cos(x)}{\cosh(x)}\Big)\Big), \quad x \in
\rbb.
   \end{align*}
Since the function $\varTheta$ is even, we
deduce that $\log \psi$ is convex if and only if
   \begin{align} \label{fdse}
\widetilde \varTheta(x) := 2 \sin(x) \tanh(x) +
\delta \Big(1 - \frac{\cos(x)}{\cosh(x)}\Big)
\Ge 0, \quad x \in \rbb_+.
   \end{align}
Note that
   \begin{align*}
1 - \frac{\cos(x)}{\cosh(x)} \Ge 0, \quad x\in
\rbb.
   \end{align*}
Thus, we have
   \begin{align} \label{ghdw}
\widetilde \varTheta(x) \Ge 0, \quad x\in
[0,\pi].
   \end{align}
Since $0 \Le \tanh(x) \Le 1$ for all $x\in
\rbb_+$, we see that $2 \sin(x) \tanh(x) \Ge -2$
for all $x\in \rbb_+$, and hence
   \begin{align}  \label{fryt}
\widetilde \varTheta(x) \overset{\eqref{fdse}}
\Ge - 2 + \delta \Big(1 -
\frac{\cos(x)}{\cosh(x)}\Big), \quad x \in
\rbb_+.
   \end{align}
It follows from $\delta \Ge \frac{2 \cosh
\pi}{\cosh \pi - 1}$ that $\delta
> 2$ and
   \begin{align*}
\cosh(x) \Ge \cosh \pi \Ge
\frac{\delta}{\delta-2} \Ge
\frac{\delta}{\delta-2} \cos(x), \quad x \in
[\pi,\infty).
   \end{align*}
Therefore, we have
   \begin{align*}
\cosh(x) - \frac{\delta}{\delta - 2} \cos(x) \Ge
0, \quad x \in [\pi, \infty),
   \end{align*}
which together with \eqref{fryt} yields
$\widetilde \varTheta(x) \Ge 0$ for all $x \in
[\pi, \infty)$. Combined with \eqref{ghdw}, this
implies \eqref{fdse}, which is equivalent to the
convexity of $\log \psi$.

Now, using \cite[Corollary~5.3]{Da-St97} (with
$\varphi$ as in the proof of Theorem~\ref{blis-zer}),
we conclude that the operator $C_{I+a,\rho}$ is
cohyponormal for every $a\in \rbb^{\kappa}$. This
together with Theorem~\ref{blis-zer} applied to $k=4$
completes the proof.
   \end{proof}
   \subsection{\label{Sec.3.4}Final remarks}
To relate our main result Theorem~\ref{wkw-uch} to the
Embry-Lambert theorem (see \cite[Theorem~2.2]{Em-Lam77B}),
let us observe that under the assumptions \eqref{miuro} and
\eqref{dens-fin}, the composition operator $C_{T,\rho}$ is
unitarily equivalent to the operator $\widetilde C_{T,\psi}$
on $L^2(\rbb^{\kappa})$ defined by
   \begin{align*}
(\widetilde C_{T,\psi} f)(x) =
\sqrt{\frac{\psi(\|Tx\|)}{\psi(\|x\|)}} \, f(Tx), \quad x\in
\rbb^{\kappa}, f\in L^2(\rbb^{\kappa}).
   \end{align*}
Indeed, this follows from the identity
   \begin{align*}
\widetilde C_{T,\psi}=W_{\psi}^{-1}C_{T,\rho} W_{\psi},
   \end{align*}
where $W_{\psi}\colon L^2(\rbb^{\kappa}) \to L^2(\mu_{\rho})$
is the unitary isomorphism defined by
$(W_{\psi}f)(x)=\sqrt{\psi(\|x\|)}\, f(x)$ for $x\in
\rbb^{\kappa}$ and $f\in L^2(\rbb^{\kappa})$. The adjoint
$\widetilde C_{T,\psi}^*$ of $\widetilde C_{T,\psi}$ takes
the following form
   \begin{align*}
(\widetilde C_{T,\psi}^*f)(x) = \frac{1}{|\det A|}
\sqrt{\frac{\psi(\|x\|)}{\psi(\|T^{-1}x\|)}} \, f(T^{-1}x),
\quad x\in \rbb^{\kappa}, f\in L^2(\rbb^{\kappa}).
   \end{align*}
In particular, if $A=I$, then
   \begin{align*}
(\widetilde C_{I+a,\psi} f)(x) =
   \begin{cases}
\sqrt{\frac{\psi(x+a)}{\psi(x)}} \, f(x+a) & \text{if
$\kappa=1$,}
   \\[2ex]
\sqrt{\frac{\psi(\|x+a\|)}{\psi(\|x\|)}} \, f(x+a) & \text{if
$\kappa >1$,}
   \end{cases}
   \quad x\in \rbb^{\kappa}, f\in L^2(\rbb^{\kappa}), a \in
\rbb^\kappa,
   \end{align*}
and
   \begin{align*}
(\widetilde S_a f)(x) =
   \begin{cases}
\sqrt{\frac{\psi(x)}{\psi(x-a)}} \, f(x-a) & \text{if $\kappa
=1$,}
   \\[2ex]
\sqrt{\frac{\psi(\|x\|)}{\psi(\|x-a\|)}} \, f(x-a) & \text{if
$\kappa
> 1$,}
   \end{cases}
\quad x\in \rbb^{\kappa}, f\in L^2(\rbb^{\kappa}), a \in
\rbb^\kappa,
   \end{align*}
where $\widetilde S_a:=\widetilde C_{I+a,\psi}^*$ for $a \in
\rbb^\kappa$. In case $\kappa=1$, this resembles the weighted
translation semigroups on $L^2(\rbb_+)$ studied by Embry and
Lambert in \cite{Em-Lam77A,Em-Lam77B} (see also
\cite{Pha-Sho23} for the multivariable case). More precisely,
for every $t\in \rbb_+$, $L^2(\rbb_+)$ is an invariant
subspace for $\widetilde S_t$, and the family $\{S_t\}_{t\in
\rbb_+}$ given by $S_t:=\widetilde S_t|_{L^2(\rbb_+)}$ is the
weighted translation semigroup with symbol
$\phi:=\sqrt{\psi}$ (notation and terminology as in
\cite{Em-Lam77B}).

Since the adjoint of a $C_0$-semigroup is a $C_0$-semigroup
(see \cite[Corollary~3.8 and Theorem~10.4]{Paz83}),
Theorem~\ref{wkw-uch} takes the following form.
   \begin{thm}
Suppose that $\psi\colon \rbb\to (0,\infty)$ is a continuous
even function and $a_0 \in \rbb^{\kappa} \backslash \{0\}$.
Then the following conditions are equivalent{\em :}
   \begin{enumerate}
   \item[(i)] $\{\widetilde S_{ta_0}\}_{t\in \rbb_+}$ is a $C_0$-semigroup
of bounded subnormal operators on $L^2(\rbb^{\kappa})$,
   \item[(ii)] $\psi$ is the hyperbolic cosine transform of a
compactly supported Borel measure on $\rbb_+$.
   \end{enumerate}
Moreover, if {\em (ii)} holds, then $\widetilde S_a$ is a
bounded subnormal operator on $L^2(\rbb^{\kappa})$ for every
$a\in \rbb^{\kappa}$.
   \end{thm}
   In the following example, we will show that considering
the hyperbolic cosine transform of measures with unbounded
supports naturally leads to unbounded composition operators
which are cosubnormal. Since in this article we deal only
with bounded operators, we circumvent the obstacles of
unboundedness by using some approach that turns out to be
equivalent to cosubnormality in the case of bounded
composition operators. For more information on unbound
subnormal composition operators, we refer the reader to
\cite[Section~7.4]{BJJS18}.
   \begin{exa} \label{invdid}
In view of Theorem~\ref{maintheorem} and Lemma~
\ref{du-du-2}, to show that the operators
$\{C_{I+a,\rho}\colon a \in \rbb^{\kappa}\}$ are
bounded and cosubnormal, it suffices to prove that if
the conditions \eqref{miuro} and \eqref{dens-fin}
hold, $\psi$ is exponentially convex and
   \begin{align} \label{ogrun}
\inf\nolimits_{b\in \rbb_+}\sup\nolimits_{x\in \rbb}
\psi(x) \E^{-b|x|} < \infty,
   \end{align}
then $\{\psi(\|x+na\|)\}_{n=0}^{\infty}$ is a
Stieltjes moment sequence for all $x,a\in
\rbb^{\kappa}$. We will show that this can be the case
if we drop the assumption \eqref{ogrun}, but at the
cost of losing the boundedness of the composition
operators under consideration.

Indeed, let us consider the function $\psi\colon \rbb
\to (0,\infty)$ defined by $\psi(x)=\E^{x^2}$ for
$x\in \rbb$. Clearly, $\psi$ is continuous and even.
Also, $\psi$ is exponentially convex because
   \begin{align*}
\sum_{j,k=1}^n \psi(x_j^* + x_k)
\lambda_j \bar\lambda_k & =
\sum_{j,k=1}^n \E^{(x_j + x_k)^2}
\lambda_j \bar\lambda_k
   \\
& = \sum_{j,k=1}^n \E^{2x_jx_k}
(\E^{x_j^2} \lambda_j)
\overline{(\E^{x_k^2}\lambda_k)}
   \\
& = \sum_{l=0}^\infty \frac{2^l}{l!}
 \sum_{j,k=1}^n (x_jx_k)^l (\E^{x_j^2}
\lambda_j)
\overline{(\E^{x_k^2}\lambda_k)}
   \\
& = \sum_{l=0}^\infty \frac{2^l}{l!}
\Big|\sum_{j=1}^n x_j^l \E^{x_j^2}
\lambda_j\Big|^2 \Ge 0,
   \end{align*}
for all finite sequence $x_1, \ldots, x_n \in
\rbb^{\kappa}$ and $\lambda_1, \ldots, \lambda_n \in
\cbb$. Clearly, $\psi$ does not satisfy the condition
\eqref{ogrun}.

Now we show that $\{\psi(\|x+na\|)\}_{n=0}^{\infty}$
is a Stieltjes moment sequence for all $x,a\in
\rbb^{\kappa}$. Fix $a \in \rbb^{\kappa} \setminus
\{0\}$. Denote by $P$ the orthogonal projection of
$\rbb^{\kappa}$ onto the one-dimensional span $\langle
a\rangle$ of $\{a\}$. Let $P^{\perp}$ stand for the
orthogonal projection of $\rbb^{\kappa}$ onto $\langle
a\rangle^{\perp}$. Then $P^{\perp}=I-P$ and $Px =
s_xa$ for $x\in \rbb^{\kappa}$, where $s_x =
\frac{\is{x}{a}}{\|a\|^2}$. Furthermore, we have
   \begin{align} \label{trzabymu}
\psi(\|x+na\|) = \E^{\|P^{\perp}x\|^2 +
\|Px + na\|^2} = \alpha
(\E^{2s_x\|a\|^2})^{n} q^{-\frac12n^2},
\quad n \in \zbb_+,
   \end{align}
where $\alpha := \E^{\|P^{\perp}x\|^2 + s_x^2\|a\|^2}$ and
$q:=\E^{-2\|a\|^2}$. It is well known that if $r\in (0,1)$,
then $\{r^{-\frac{1}{2}n^2}\}_{n=0}^{\infty}$ is an
indeterminate Stieltjes moment sequence (see \cite[Sect.\
2]{Berg98}; see also \cite[pp.\ J.106 and J.107]{Sti94-95}).
This together with \eqref{trzabymu} implies that the sequence
$\{\psi(\|x+na\|)\}_{n=0}^{\infty}$ is a Stieltjes moment
sequence for all $x,a\in \rbb^{\kappa}$. However, in view of
\cite[Theorem~2.2]{Da-St97}, the operators
$\{C_{I+a,\rho}\colon a \in \rbb^{\kappa}\backslash \{0\}\}$
are not bounded on $L^2(\mu_{\rho})$.
   \hfill{$\diamondsuit$}
   \end{exa}

   \appendix
\numberwithin{equation}{section}
   \section{}
We give another proof of the implication
(ii)$\Rightarrow$(i) of Theorem~\ref{maintheorem}
which uses the characterization of the Laplace
transform (see \cite[Corollary~4.4.5]{b-ch-r}).
   \begin{proof}[Proof II of implication
{\rm (ii)$\Rightarrow$(i)} of
Theorem~\ref{maintheorem}] By (ii-d), we have
   \begin{align*}
\psi(x-n)\E^{-b(x-n)} \Le a \E^{2nb}, \quad x\in \rbb_+, \,
n\in \zbb_+.
   \end{align*}
Therefore, according to (ii-a), (ii-b) and
Lemma~\ref{kiedy-cosub-cosh}, for every $n\in \zbb_+$, the
function $\rbb_+ \ni x \longmapsto \psi(x-n) \E^{-b(x-n)} \in
\rbb_+$ is bounded, continuous and positive definite on the
$*$-semigroup $(\rbb_+,+,x^*=x)$. In view of \cite[Corollary~
4.4.5]{b-ch-r}, for every $n\in \zbb_+$, there exists a
finite Borel measure $\mu_{n}$ on $\rbb_+$ such that
   \begin{align} \label{psi-e}
\psi(x-n) \E^{-b(x-n)} = \int_{\rbb_+} \E^{-xu} \D
\mu_{n} (u), \quad x\in \rbb_+, \, n\in \zbb_+.
   \end{align}
As a consequence, we have
   \begin{align*}
\int_{\rbb_+} \E^{-xu} \D \mu_{n} (u) &
\overset{\eqref{psi-e}}= \psi(x-n) \E^{-b(x-n)}
   \\
& \hspace{1.25ex} = \psi((x+1)-(n+1))
\E^{-b((x+1)-(n+1))}
   \\
& \overset{\eqref{psi-e}} = \int_{\rbb_+} \E^{-xu}
\E^{-u} \D \mu_{n+1}(u), \quad x\in \rbb_+, \, n\in
\zbb_+.
   \end{align*}
Since the Borel measure $\E^{-u} \D \mu_{n+1}(u)$ is
finite on $\rbb_+$, the injectivity of the Laplace
transform (see \cite[Proposition~ 4.4.2]{b-ch-r})
implies that
   \begin{align*}
\mu_{n}(\varDelta) = \int_{\varDelta} \E^{-u} \D
\mu_{n+1}(u), \quad \varDelta\in \borel{\rbb_+}, \,
n\in \zbb_+,
   \end{align*}
or equivalently that
   \begin{align*}
\mu_{n+1}(\varDelta) = \int_{\varDelta} \E^{u} \D
\mu_{n}(u), \quad \varDelta\in \borel{\rbb_+}, \, n\in
\zbb_+.
   \end{align*}
By induction on $n$, we have
   \begin{align} \label{psi-e2}
\mu_{n}(\varDelta) = \int_{\varDelta} \E^{nu} \D
\mu_{0}(u), \quad \varDelta\in \borel{\rbb_+}, \, n\in
\zbb_+.
   \end{align}
It follows from \eqref{psi-e} and \eqref{psi-e2} that
   \begin{align*}
\psi(x-n) \E^{-b(x-n)} = \int_{\rbb_+} \E^{-(x-n)u} \D
\mu_{0}(u), \quad x\in \rbb_+, \, n\in \zbb_+.
   \end{align*}
Since for any $y\in \rbb$ there exist $x\in \rbb_+$
and $n\in \zbb_+$ such that $y=x-n$, we get
   \begin{align} \label{psy2}
\psi(y) \E^{-by} = \int_{\rbb_+} \E^{-yu} \D
\mu_{0}(u), \quad y\in \rbb.
   \end{align}
Fix any $\tilde y\in \rbb$ such that $\tilde y < 0$.
Then we have
   \begin{align*}
\min\Big\{\alpha\in \rbb_+\colon & \mu_0\big(\{u\in
\rbb_+\colon \E^{-\tilde yu} > \alpha\}\big)=0\Big\}
\overset{(*)} = \lim_{n\to \infty} \bigg(\int_{\rbb_+}
(\E^{-\tilde yu})^n \D \mu_{0}(u)\bigg)^{1/n}
   \\
& \overset{\eqref{psy2}}= \lim_{n\to \infty}
\psi(n\tilde y)^{1/n} \E^{-b\tilde y}
\overset{\eqref{coshtr2}} \Le \lim_{n\to \infty}
a^{1/n} \E^{2b|\tilde y|} \Le \E^{2b|\tilde y|},
   \end{align*}
where $(*)$ follows from \cite[p.\ 71, Exercise~
4]{rud87}. Therefore,
   \begin{align*}
0=\mu_0\big(\{u\in \rbb_+\colon \E^{-\tilde yu} >
\E^{2b|\tilde y|}\}\big) = \mu_0\big((2b,\infty)\big),
   \end{align*}
or equivalently $\supp \mu_0 \subseteq [0,2b]$. Hence,
by the measure transport theorem, we~ have
   \allowdisplaybreaks
   \begin{align*}
\psi(y) & \overset{\text{(ii-c)}} = \frac 12
\Big(\psi(y) + \psi(-y)\Big)
   \\
& \overset{\eqref{psy2}}= \frac 12 \Big(\int_{[0,2b]}
\E^{-y(u-b)} \D\mu_0(u) + \int_{[0,2b]} \E^{y(u-b)}
\D\mu_0(u)\Big)
   \\
& \hspace{1.2ex} = \int_{[0,2b]} \cosh(y(u-b))
\D\mu_0(u)
   \\
& \hspace{1.2ex} = \int_{[-b,b]} \cosh(yu)
\D\mu_{*}(u)
   \\
& \hspace{1.2ex} = \int_{[-b,0)} \cosh(y(-u))
\D\mu_{*}(u) + \mu_{*}(\{0\}) + \int_{(0,b]} \cosh(yu)
\D\mu_{*}(u)
   \\
& \hspace{1.2ex} = \int_{(0, b]} \cosh(yu)
\D\tilde\mu_{*}(u) + \mu_{*}(\{0\}) + \int_{(0,b]}
\cosh(yu) \D\mu_{*}(u)
   \\
& \hspace{1.2ex} = \int_{[0, b]} \cosh(yu) \D\nu(u),
\quad y \in \rbb,
   \end{align*}
where $\mu_{*}\colon \borel{[-b,b]} \to \rbb_+$,
$\tilde\mu_{*}\colon \borel{(0,b]}\to \rbb_+$ and
$\nu\colon \borel{[0,b]} \to \rbb_+$ are finite Borel
measures defined by
   \begin{align*}
\mu_{*}(\varDelta) & = \mu_0(b+\varDelta), \quad
\varDelta \in \borel{[-b,b]},
   \\
\tilde\mu_{*}(\varDelta) & = \mu_{*}(-\varDelta),
\quad \varDelta \in \borel{(0,b]},
   \\
\nu(\varDelta) & = \tilde \mu_{*}(\varDelta \cap
(0,b]) + \mu_{*}(\{0\}) \chi_{\varDelta}(0) +
\mu_{*}(\varDelta \cap (0,b]), \quad \varDelta \in
\borel{[0,b]}.
   \end{align*}
This yields (i).
   \end{proof}
   \bibliographystyle{amsalpha}
   
   \end{document}